\documentclass[11pt]{article}
\usepackage[latin1]{inputenc}
\usepackage{amsmath,amsthm,amssymb}
\usepackage{amsfonts}
\usepackage{amsmath,amsthm,amssymb,amscd}
\usepackage{latexsym}
\usepackage{color}
\usepackage{graphicx}
\usepackage{mathrsfs}
\usepackage{cite}

\makeatletter \@addtoreset{equation}{section} \makeatother

\newtheorem{theorem}{Theorem}[section]
\newtheorem{definition}{Definition}[section]
\newtheorem{proposition}{Proposition}[section]
\newtheorem{lemma}{Lemma}[section]
\newtheorem{remark}{Remark}[section]

\newtheorem{corollary}[theorem]{Corollary}

\begin{document}
\title{Normalized ground states for the critical fractional NLS equation \\with a perturbation}
\author{ Maoding Zhen$^{a}$,  Binlin Zhang$^{b, }$\footnote{Corresponding author.
E-mail address:\, {\tt maodingzhen@163.com} (M. Zhen), \ {\tt zhangbinlin2012@163.com} (B. Zhang)}\\
\footnotesize $^a\,$School of Mathematics, Hefei University of Technology, Hefei 230009, P.R. China\\
\footnotesize $^b\,$College of Mathematics and Systems Science,
Shandong University of Science and Technology,
\\ \footnotesize Qingdao, 266590, P.R. China}
\date{ }
\maketitle
\begin{abstract}
{In this paper, we study normalized ground states for the following critical fractional NLS equation with prescribed mass:
\begin{equation*}
\begin{cases}

(-\Delta)^{s}u=\lambda u +\mu|u|^{q-2}u+|u|^{2_{s}^{\ast}-2}u,&x\in\mathbb{R}^{N},  \\

\int_{\mathbb{R}^{N}}u^{2}dx=a^{2},\\
\end{cases}
\end{equation*}
where $(-\Delta)^{s}$ is the fractional Laplacian, $0<s<1$, $N>2s$, $2<q<2_{s}^{\ast}=2N/(N-2s)$ is a fractional critical Sobolev exponent, $a>0$, $\mu\in \mathbb{R}$. By using Jeanjean's trick in \cite{Jeanjean}, and the standard method which can be found in \cite{Brezis} to overcome the lack of compactness, we first prove several existence and nonexistence results for a $L^{2}$-subcritical (or $L^{2}$-critical or $L^{2}$-supercritical) perturbation $\mu|u|^{q-2}u$, then we give some results about the behavior of the ground state obtained above  as $\mu\rightarrow 0^{+}$.
Our results extend and improve the existing ones in several directions.
}\medskip

\emph{\bf Keywords:}  Fractional Laplacian; Critical exponent; Normalized ground state solution.\medskip

\emph{\bf 2010 Mathematics Subject Classification:} 35J20, 35B33,  58E05.
\end{abstract}

\section{Introduction and main results}
\hspace{0.6cm}
In this paper, we consider the following critical  nonlinear Schr\"{o}dinger equation involving the fractional Laplacian:
\begin{align}\label{int1}
(-\Delta)^{s}u=\lambda u +\mu|u|^{q-2}u+|u|^{2_{s}^{\ast}-2}u&\ \ x\in\mathbb{R}^{N},
\end{align}
and possessing prescribed mass
\begin{align}\label{in2}
 \int_{\mathbb{R}^{N}}u^{2}dx=a^{2},
\end{align}
where $(-\Delta)^{s}$ is the fractional Laplacian, $0<s<1,\ 2<q<2_{s}^{\ast}=2N/(N-2s)$ is a fractional critical Sobolev exponent.
The fractional Laplacian\ $(-\Delta)^{s}$\ is defined by
$$
(-\Delta)^{s}u(x)=C(N,s) {\rm P.V.}\int_{\mathbb{R}^{N}}\frac{u(x)-u(y)}{|x-y|^{N+2s}}dy,\ \ x\in\mathbb{R}^{N}
$$
for $u \in C^\infty_0(\mathbb{R}^{N})$, where $C(N,s)$ is a suitable positive normalizing constant and P.V. denotes the Cauchy principle value. We refer to \cite{DPV, MRS, GDGC,FV14,EV09, BV16} for a simple introduction to basic properties of the fractional Laplace operator and concrete applications based on variational methods.

Our main driving force for the study of
\eqref{int1} arises in the study of the following time-dependent fractional Schr\"odinger
equation with combined power nonlinearities:
\begin{align}\label{eq1.1}
i\psi_t-(-\Delta)^{s}\psi+\mu|\psi|^{q-2}\psi+|\psi|^{2_{s}^{\ast}-2}\psi=0\ \ \mbox{in}\ \mathbb{R}^{N}.
\end{align}
When searching for stationary
waves of the form $\psi(t,x)=e^{-i\lambda t}u(x)$, where $\lambda\in \mathbb{R}$ is the chemical
potential and $u(x):\mathbb{R}^N\rightarrow \mathbb{C}$ is a time-independent function, one
is led to studying \eqref{int1}.
In this case, particular attention is paid to {\it ground state solutions}, a.e., solutions minimizing an energy functional
 among all non-trivial solutions. An alternative choice is to look for solutions to \eqref{int1} having prescribed mass,
and in this case $\lambda \in \mathbb{R}$ is part of the unknown. This approach seems particularly meaningful from the physical
point of view, since, in addition to being a conserved quantity for the time dependent \eqref{eq1.1},
the mass has often an evident physical meaning, for example, it indicates the power supply
in nonlinear optics, or the total number of atoms in Bose-Einstein condensation. Moreover,
this approach gives a better insight of the properties of the stationary solutions
for \eqref{int1}, for example, stability or instability, see \cite{Soave} for more details.

The existence of normalized stationary states can be summarized as follows: given $a>0$ and $\mu\in \mathbb{R}$,
$2<q<2_s^*$, our aim is to find $(\lambda, u)\in \mathbb{R}\times H^{s}(\mathbb{R}^N, \mathbb{C})$ such that
\eqref{int1} and \eqref{in2} hold.
 For the Laplacian case, i.e., $s=1$ in \eqref{int1}, we would like to mention a seminal paper by Jeanjean in \cite{Jeanjean}, which dealt with the existence of normalized solutions when the energy function is unbounded from below on the $L^{2}$ constraint.
 In fact, the  normalized solutions for nonlinear Schr\"{o}dinger equation or system have attracted much attention in recent years, both for their interesting theoretical structure and their concrete applications (see \cite{BJJN,TBNS,TBNSS,Soave,Soave1} and references therein). Since $\lambda$ and $\mu$  are parts of the unknown, the Nehari manifold method is not available in the framework of normalized solutions. Meanwhile, the appearance of the  $L^{2}$ constraint makes some classical methods, used to prove the boundedness of any Palais-Smale sequence for the unconstrained problem, difficult to implement. It is well known that a new $L^{2}$-critical exponent $\widetilde{p}=2+4/N$ plays a special role. Indeed, if the problem is $L^{2}$-subcritical, i.e., $2<q<p<\widetilde{p}$, the energy  functional $E_{\mu}$ (defined in \eqref{eq1.2}) is bounded from below on the constraint $\overline{S}_{a}=\{u\in H^1(\mathbb{R}^N, \mathbb{C}): \int_{\mathbb{R}^{N}}u^{2}dx=a^{2}\}$, so the ground state solution can be found as global minimizers of $E_{\mu}|_{\overline{S}_{a}}$. Moreover, if the problem is $L^{2}$-supercritical, i.e., $\widetilde{p}<q<p<2^{\ast}=2N/(N-2)$, then the energy  functional $E_{\mu}$ is unbounded both from
above and from below on $\overline{S}_{a}$. In this case, the ideas introduced by  Jeanjean in \cite{Jeanjean}  can be employed to consider the existence of normalized solutions for any $a,\mu >0.$

 Compared to the semilinear case that corresponds to the Laplace operator, the fractional Laplacian problems are nonlocal and more challenging. For fractional Laplacian equations or systems with fixed $\lambda_{i}$, the existence and non-degeneracy of solutions have been studied by a lot of researchers and there are many results about existence, nonexistence, multiplicity of solutions for fractional Laplacian equation, since it seems almost impossible for us to provided a complete list of references, we just refer the readers to  \cite{CS1,CS2,BCSS,CRS,CDS2,CT,GLZ,XSZ,HPP,ZZ1,ZZ2,XZR,RLFL,FLS} and references therein.

Recently, Soave in \cite{Soave, Soave1} first investigated the existence and properties of ground states for the nonlinear Schr\"{o}dinger equation  with combined power type nonlinearities and also gave new criteria for global existence and finite time blow-up in the associated dispersive equation.
More precisely, Soave in \cite{Soave} considered the normalized solutions for subcritical exponent and give a complete classification about the existence and nonexistence of normalized solution for $L^{2}$-subcritical, $L^{2}$-critical and $L^{2}$-supercritical. For the critical case, the problem is also interesting and challenging. By focusing the leading nonlinearity and analysing how the introduction of lower order term modifies the energy functional structure,  Soave in \cite{Soave1} obtained the existence and nonexistence of normalized solutions for $L^{2}$-subcritical, $L^{2}$-critical and $L^{2}$-supercritical in the Sobolev critical case. Due to the lack of compactness of the Sobolev embedding $H^{1}(\mathbb{R}^{N})\hookrightarrow L^{2^{\ast}}(\mathbb{R}^{N})$, the problem is more complicated, however, the difficulty was overcome ingeniously by combining some ideas from \cite{Brezis} and \cite{Jeanjean}.

Inspired by the above-mentioned works, especially by \cite{Soave, Soave1}, in the present paper our goal is two-fold.
One is to show the existence and nonexistence of normalized ground states for fractional elliptic equations with critical exponent. Another  is to give  some results about the behavior of  ground state solutions obtained above  as $\mu\rightarrow 0^{+}$. The method we use is Jeanjean's method \cite{Jeanjean} combined with Pohozaev manifold argument. By using the test function as in \cite{SV1}, we show that the least energy of the equation is below the critical energy $\frac{s}{N}S^{N/(2s)}_{s}$ under the proper conditions given on $N, s,p,\lambda$ under which the Palais-Smale condition is satisfied. The main difficulty is to prove the convergence of constrained Palais-Smale sequence. Indeed,  if we find a bounded Palais-Smale sequence, according to the compactness of the embedding $H^{s}_{rad}(\mathbb{R}^{N})\hookrightarrow L^{p}(\mathbb{R}^{N}),\ 2<p<2_{s}^{\ast}$, we just get a strongly convergent subsequence in $L^{p}(\mathbb{R}^{N})$, but we cannot deduce the strong convergence in $L^{2}(\mathbb{R}^{N}).$ Hence we require new arguments to overcome the lack of compactness of the embedding $H^{s}_{rad}(\mathbb{R}^{N})\hookrightarrow L^{2}(\mathbb{R}^{N}).$ To this end, we adopt some ideas of \cite{PFAQ} to obtain a Liouville-type result.

Before we state our main results, we first introduce some notations. Let ${H^{s}(\mathbb{R}^{N})}$ be the Hilbert space of function in $\mathbb{R}^N$ endowed with the standard inner product and norm
$$\langle u,v\rangle=\int_{\mathbb{R}^{N}}\left((-\Delta)^{\frac{s}{2}} u(-\Delta)^{\frac{s}{2}}v+uv\right)dx,\ \ \|u\|_{H^{s}(\mathbb{R}^{N})}^{2}=\langle u,u\rangle.$$

Let $D_{s}(\mathbb{R}^N)$ be the Hilbert space defined as the completion of $C_{c}^{\infty}(\mathbb{R}^N)$ with the inner product
$$
\langle u,v\rangle_{D_{s}(\mathbb{R}^N)}=\frac{C(N,s)}{2}\iint_{\mathbb{R}^{2N}}\frac{(u(x)-u(y))(v(x)-v(y))}{|y-x|^{N+2s}}dxdy
$$and norm
$$
\|u\|^{2}_{D_{s}(\mathbb{R}^{N})}=\int_{\mathbb{R}^{N}}|(-\Delta)^{\frac{s}{2}}u|^{2}dx=\frac{C(N,s)}{2}\iint_{\mathbb{R}^{2N}}\frac{|u(x)-u(y)|^{2}}{|y-x|^{N+2s}}dxdy.
$$

The energy functional associated with \eqref{int1} and the constraint are given by
\begin{align}\label{eq1.2}
E_{\mu}(u)&=\frac{1}{2}||u||^{2}_{D_{s}(\mathbb{R}^{N})}
-\frac{\mu}{q}\int_{\mathbb{R}^{N}}|u|^{q}dx-\frac{1}{2_{s}^{\ast}}\int_{\mathbb{R}^{N}}|u|^{2_{s}^{\ast}}dx
\end{align}
and
$$S_{a}=\left\{u\in H^s(\mathbb{R}^N, \mathbb{C}): \int_{\mathbb{R}^{N}}u^{2}dx=a^{2}\right\}.$$

Let $S_s$ be the sharp imbedding constant of ${D_{s}(\mathbb{R}^{N})}\hookrightarrow L^{2^{\ast}}(\mathbb{R}^{N}),$
\begin{align}\label{int2}
S_{s}=\inf \limits_{u\in {D_{s}(\mathbb{R}^{N})})\setminus \{0\}}\frac{\|u\|^{2}_{D_{s}(\mathbb{R}^{N})}}{(\int_{\mathbb{R}^{N}}|u|^{2^{\ast}}dx)^{\frac{2}{2^{\ast}}}}.
\end{align}
From \cite{CT} we know that  $S_{s}$\ is attained in\ $\mathbb{R}^{N}$ by
\begin{equation}\label{G1}
U_{\epsilon,x_{0}}(x)=\kappa(\epsilon^{2}+|x-x_{0}|^{2})^{-\frac{N-2s}{2}}
\end{equation}
 where $\kappa\neq 0 \in \mathbb{R},\ \varepsilon>0 $ are fixed constants and $x_{0}\in
 \mathbb{R}^{N}$.

To present our main results, put
\begin{align*}
\gamma_{q,s}=\frac{N(q-2)}{2qs},
\end{align*}
\begin{align}\label{equ2}
C'=\frac{q(2_{s}^{\ast}-2)}{2C^{q}_{N,q,s}(2_{s}^{\ast}-q\gamma_{q,s})}\left(\frac{(2-q\gamma_{q,s})2_{s}^{\ast}}{2(2_{s}^{\ast}-q\gamma_{q,s})}S^{\frac{2_{s}^{\ast}}{2}}_{s}\right)^{\frac{2-q\gamma_{q,s}}{2_{s}^{\ast}-2}},
\end{align}
\begin{align}\label{equ3}
C''=\frac{22_{s}^{\star}}{(2_{s}^{\star}-q\gamma_{q,s})C^{q}_{N,q,s}}\left(\frac{Nq\gamma^{2}_{q,s}S^{\frac{N}{2s}}_{s}}{(2-q\gamma_{q,s})s}\right)^{\frac{2-q\gamma_{q,s}}{2}}.
\end{align}
\begin{theorem}\label{TH1}
Let $N>2s,\ a,\mu>0$  and $2<q<\overline{p}:=2+4s/N$. If there exists a constant $\alpha=\alpha(N,q)>0$ such that
\begin{align}\label{equ4}
\mu a^{q(1-\gamma_{q,s})}<\alpha:=\min\{C', C''\},
\end{align}
then $E_{\mu}|_{S_{a}}$ has a ground state $\widetilde{u}$ with the following properties: $\widetilde{u}$ is a positive, radially symmetric function and solves \eqref{int1}--\eqref{in2} for some $\widetilde{\lambda}<0.$ Moreover, $m(a,\mu)<0$ and $\widetilde{u}$ is an interior local minimizer of $E_{\mu}(u)$ on the set $$A_{k}=\{u\in S_{a}: ||u||_{D_{s}(\mathbb{R}^{N})}< k\},$$ for suitable $k$ small enough. Any other ground state solution of $E_{\mu}$ on $S_{a}$ is a local minimizer of $E_{\mu}$ on $A_{k}$.
\end{theorem}
\begin{theorem}\label{TH2}
Let $N>2s,\ a,\mu>0$  and $2<q=\overline{p}$. If
\begin{align}\label{U200}
\mu a^{\frac{4s}{N}}< \overline{p}\left(2C^{\overline{p}}_{N,\overline{p},s}\right)^{-1},
\end{align}
then $E_{\mu}|_{S_{a}}$ has a ground state $\widetilde{u}$ with the following properties: $\widetilde{u}$ is a positive, radially symmetric function and solves \eqref{int1}--\eqref{in2} for some $\widetilde{\lambda}<0.$ Moreover, $0<m(a,\mu)<\frac{s}{N}S^{N/(2s)}_{s},$ and $\widetilde{u}$ is a critical point of  Mountain Pass type.
\end{theorem}
\begin{theorem}\label{TH3}
Let $N>2s,\ a,\mu>0$ and $\overline{p}<q<2_{s}^{\ast}$. If one of the following conditions holds:
\begin{enumerate}
 \item[$(1)$] $ N>4s\  \text{and}\  \mu a^{q(1-\gamma_{q,s})}<\frac{S^{\frac{N}{4s}q(1-\gamma_{q,s})}_{s}}{\gamma_{q,s}},$

 \item[$(2)$] $ N=\frac{q}{q-1}2s\ \text{and}\ \mu a^{q(1-\gamma_{q,s})}<\frac{S^{\frac{N}{4s}q(1-\gamma_{q,s})}_{s}}{\gamma_{q,s}},$

 \item[$(3)$] $N=4s \  \text{or}\ \frac{q}{q-1}2s<N<4s\  \text{or}\ 2s<N<\frac{q}{q-1}2s, $
 \end{enumerate}
then $E_{\mu}|_{S_{a}}$ has a ground state $\widetilde{u}$ with the following properties: $\widetilde{u}$ is a positive, radially symmetric function and solves \eqref{int1}--\eqref{in2} for some $\widetilde{\lambda}<0.$ Moreover, $0<m(a,\mu)<\frac{s}{N}S^{N/(2s)}_{s},$ and $\widetilde{u}$ is a critical point of  Mountain Pass type.
\end{theorem}

\begin{theorem}\label{Lem1}
Let $a>0$ and $\mu=0$. Then we have the following conclusions:
\begin{enumerate}
\item[$(1)$] If $N>4s$, then $E_{0}$ on $S_{a}$ has a unique positive radial ground state $U_{\epsilon,0}$ defined in \eqref{G1} for the unique choice of $\epsilon>0$ which gives $||U_{\epsilon,0}||_{L^{2}(\mathbb{R}^{N})}=a.$

 \item[$(2)$] If $2s<N\leq4s$, then \eqref{int1} has no positive solutions in $S_{a}$ for any $\lambda\in \mathbb{R}$.
\end{enumerate}
\end{theorem}
\begin{theorem}\label{TH4}
Let $u_{\mu}$ be the corresponding positive ground state solution obtained in Theorems \ref{TH1}--\ref{TH3} with energy level $m(a,\mu)$.
Then the following conclusions hold:
\begin{enumerate}
\item[$1)$] If $2<q<\overline{p}$, then $m(a,\mu)\rightarrow 0$, and $\|u_{\mu}\|^{2}_{D_{s}(\mathbb{R}^{N})}\rightarrow 0$ as $\mu\rightarrow 0^{+}$.
\item[$2)$] If $\overline{p}\leq q<2^{\ast}$, then $m(a,\mu)\rightarrow \frac{s}{N}S^{\frac{N}{2s}}_{s}$, and $\|u_{\mu}\|^{2}_{D_{s}(\mathbb{R}^{N})}\rightarrow S^{\frac{N}{2s}}_{s}$ as $\mu\rightarrow 0^{+}$.
\end{enumerate}
\end{theorem}
\begin{remark}
{\rm The assumptions \eqref{equ2}, \eqref{equ3} and \eqref{U200} are used to describe the geometry of $E_{\mu}$. Meanwhile, the assumptions in Theorem \ref{TH3} are applied to overcome the lack of compactness.}
\end{remark}
\begin{remark}
{\rm We should point out that Luo and Zhang in \cite{LZ20} considered the subcritical fractional equation with combined nonlinearities and proved the existence and nonexistence of normalized solutions, however, in this paper we consider the existence and nonexistence of normalized solutions for the critical fractional equation with combined nonlinearities. Compared with the subcritical case, the critical case is more complicated and needs to overcome the lack of compactness.}
\end{remark}
 In this paper, we invoke some ideas proposed by Soave in \cite{Soave,Soave1}. Compared to the Laplacian problems, the fractional Laplacian problems are nonlocal and more challenging. Indeed, when we consider the fractional Laplacian problem, the corresponding algebraic equation is about fractional order, which is more complicated to deal with than an integer-order algebraic equation. Moreover, one of the main difficulties is to analyze the convergence of constrained Palais-Smale sequence.  To overcome the lack of compactness, we employ delicate methods which can be found in \cite{Brezis}, that is, cut-off technique and energy estimate.  To show the least energy strictly less than ``the threshold energy", our analysis is more difficult and complicated. Indeed, when we deal with the $L^{2}$-critical and $L^{2}$-supercritical case, we need to give an exact classification of dimensions $N$, which depends on $s$ and $q$, and give exact estimate for energy function in each cases. In particular, for $2s<N<4s$ and $\mu=0$, to show Theorem \ref{Lem1} (2), we need to prove that all solutions for equation $(-\Delta)^{s}v=v^{2_{s}^{\ast}-1}, \ v\geq 0\ \text{in}\ \mathbb{R}^{N},$ must be $\alpha U_{\epsilon,0}$ for some $\alpha,\epsilon>0$.

Finally, let us sketch the proof of above theorems. In general, this study can be considered as a counterpart of the fractional Br\'ezis-Nirenberg  problem in the context of normalized solutions. To overcome the lack of compactness which is a crucial step for the critical case, we show that the least energy strictly less than ``the threshold energy", we employ delicate methods which can be found in \cite{Brezis}, that is, cut-off technique and energy estimate. The convergence of Palais-Smale sequence (see Proposition \ref{LM5}) is one of the most delicate ingredient in the proofs of our main results. We introduce a fiber maps $\Psi^{\mu}_{u}(t)$ (see \eqref{pp}), it is well known that any critical point of  $E_{\mu}|_{S_{a}}$ stays in $\mathcal{P}_{a,\mu}$(see \eqref{k1}), the monotonicity and convexity properties of $\Psi^{\mu}_{u}(t)$ strongly affect the structure of $\mathcal{P}_{a,\mu}$. It is easy to see that $(\Psi^{\mu}_{u})'(t)=P_{\mu}(t\star u)$, so that $t$ is a critical point of $\Psi^{\mu}_{u}(t)$ if and only if $t\star u\in \mathcal{P}_{(a,\mu)}$ and in particular $u\in \mathcal{P}_{(a,\mu)}$ if and only if $0$ is a critical point of $\Psi^{\mu}_{u}(t)$. In this spirit, we split $\mathcal{P}_{a,\mu}$ into three parts, then we prove that $\mathcal{P}^{0}_{a,\mu}=\emptyset$ and $\mathcal{P}_{a,\mu}$ is a smooth manifold of codimension 1 in $S_{a}$ under suitable conditions. For $L^{2}$-subcritical case, we restricted the energy function $E_{\mu}$ on the $\mathcal{P}_{a,\mu}$ and we can prove that  $E_{\mu}|_{\mathcal{P}_{a,\mu}}$ is  bounded from below, so a local minimizer $\widetilde{u}$ for $E_{\mu}$ on the $\mathcal{P}_{a,\mu}$ can be obtained. For $L^{2}$-critical/supercritical, we construct different linking structures to obtain the Mountain Pass type solutions.

The paper is organized as follows. In Section \ref{sec1}, we introduce some preliminaries that will be used to prove Theorems \ref{TH1}--\ref{TH3}. In Section \ref{sec2}, we give some lemmas for $L^{2}$-subcritical perturbation. In Section \ref{sec3}, we give some preliminaries for $L^{2}$-critical perturbation. In Section \ref{sec4}, we give some lemmas for $L^{2}$-supercritical perturbation. In Section \ref{sec5}, we prove Theorem \ref{TH1}.  In Section \ref{sec6}, we prove Theorems \ref{TH2}--\ref{TH3}. In Section \ref{Lab5}, we prove Theorem \ref{Lem1}.  Finally, the proof of Theorem \ref{TH4} will be given in Section \ref{sec7}.

\section{Preliminaries }\label{sec1}
Let $S_s$ be the sharp embedding constant of ${D^{s}(\mathbb{R}^{N})}\hookrightarrow L^{2_{s}^{\ast}}(\mathbb{R}^{N}),$
\begin{align}
S_{s}=\inf \limits_{u\in {D^{s}(\mathbb{R}^{N})})\setminus \{0\}}\frac{\|u\|^{2}_{D^{s}(\mathbb{R}^{N})}}{(\int_{\mathbb{R}^{N}}|u|^{2_{s}^{\ast}}dx)^{\frac{2}{2_{s}^{\ast}}}}.
\end{align}from \cite{CT} $S_{s}$\ is attained in\ $\mathbb{R}^{N}$ by \ $\widetilde{u}(x)=\kappa(\varepsilon^{2}+|x-x_{0}|^{2})^{-\frac{N-2s}{2}}$, where $\kappa\neq 0 \in \mathbb{R},\ \varepsilon>0 $ are fixed constants and $x_{0}\in
 \mathbb{R}^{N}$.

It is useful to introduce the  fractional Gagliardo-Nirenberg-Sobolev inequality (see\cite{FLS})
\begin{equation}\label{4}
\int_{\mathbb{R}^{N}}|u|^{p}dx\leq C_{N,p,s}\left(\int_{\mathbb{R}^{N}}|(-\Delta)^{\frac{s}{2}}u|^{2}dx\right)^{\frac{N(p-2)}{4s}}\left(\int_{\mathbb{R}^{N}}|u|^{2}dx\right)^{\frac{p}{2}-\frac{N(p-2)}{4s}}\ \ \text{for all} \ \ u\in H^{s}(\mathbb{R}^{N}).
\end{equation}

Define
\begin{align*}
\gamma_{p,s}=\frac{N(p-2)}{2ps},
\end{align*}
 it is easy to see that
\begin{equation}\label{int45}
p\gamma_{p,s}\left\{\begin{array}{ll}
                <{\displaystyle 2},& \text{if} \ 2<p<\overline{p},\\
                ={\displaystyle 2},& \text{if} \ p=\overline{p},\\
                >{\displaystyle 2},& \text{if} \ \overline{p}<p<2_{s}^{\ast},
           \end{array}\right.
\text{and that}\ \
\gamma_{2_{s}^{\ast}}=1,
\end{equation}
and
\begin{align}\label{U3}
  \|u\|_{L^{p}}\leq C_{N,p,s}\|(-\Delta)^{s}u\|^{\gamma_{p,s}}_{L^{2}}\|u\|^{1-\gamma_{p,s}}_{L^{2}}\ \text{for all} \ \ u\in H^{s}(\mathbb{R}^{N}).
\end{align}

We first give  the following key  Pohozaev identity for the fractional Laplace operator.

\begin{proposition}[Theorem A.1 in \cite{BHAKTA}]\label{pro}
 Let $u\in H^{s}(\mathbb{R}^{N})\bigcap L^{\infty}(\mathbb{R}^{N})$ be a positive solution of  $(-\Delta)^{s}u= f(u)$ and $F(u)\in L^{1}(\mathbb{R}^{N})$, then it hold that $$\frac{N-2s}{2}\int_{\mathbb{R}^{N}}|(-\Delta)^{\frac{s}{2}}u|^{2}dx=N\int_{\mathbb{R}^{N}}F(u)dx,$$ where $F(u)=\int^{u}_{0}f(t)dt$.
 \end{proposition}
\begin{remark}
{\rm Since $u\in H^{s}(\mathbb{R}^{N})$, by the fractional Sobolev embedding theorem (see \cite[Theorem 2.2]{Sec16}), it is easy to see $u\in L^{p}(\mathbb{R}^{N}),\ p\in [2,2_{s}^{\ast}],$ which implies that  $F(u)=\frac{\lambda}{2} u^{2} +\frac{\mu}{q}|u|^{q}+\frac{1}{2_{s}^{\ast}}|u|^{2_{s}^{\ast}}\in L^{1}(\mathbb{R}^{N})$, hence we can modify  the proof of Proposition 5.1 in \cite{BCPS} to obtain that $u\in L^{\infty}B(0,\frac{r}{2})$. Using the same arguments for a neighborhood of any $x\in \mathbb{R}^{N}$, we get $u\in L_{loc}^{\infty}(\mathbb{R}^{N})$. Thus we can use similar arguments as in the proof of Theorem 3.4 in \cite{PAT12} to obtain that $u\in L^{\infty}(\mathbb{R}^{N})$, which implies that the above Pohozaev identity can be applied to our equation. In fact, similar method to prove $u\in L^{\infty}(\mathbb{R}^{N})$ can also be used to prove Proposition 4.1 in \cite{CW13}.}
\end{remark}
 \begin{lemma}\label{Lemp}
Let  $u\in H^{s}(\mathbb{R}^{N})$ is a solution of \eqref{int1}, then
\begin{equation}\label{k1}
\mathcal{P}_{a,\mu}=\{u\in S_{a}:P_{\mu}(u)=0\},
\end{equation}
where $$P_{\mu}(u)=s||u||^{2}_{D_{s}(\mathbb{R}^{N})}-\mu\gamma_{q,s}s\int_{\mathbb{R}^{N}} |u|^{q}dx-s\int_{\mathbb{R}^{N}} |u|^{2_{s}^{\ast}}dx.$$
\end{lemma}
\begin{proof}
From Proposition \ref{pro}, we have
\begin{equation}\label{intk}
\frac{N-2s}{2}||u||^{2}_{D_{s}(\mathbb{R}^{N})}=\lambda\frac{ N}{2}\int_{\mathbb{R}^{N}} u^{2}dx+\frac{ N\mu }{q}\int_{\mathbb{R}^{N}} |u|^{q}dx+\frac{ N }{2_{s}^{\ast}}\int_{\mathbb{R}^{N}} |u|^{2_{s}^{\ast}}dx.
\end{equation}
 Since $u$ is a solution of \eqref{int1}, we have
\begin{equation}\label{intp}
||u||^{2}_{D_{s}(\mathbb{R}^{N})}=\lambda\int_{\mathbb{R}^{N}} u^{2}dx+\mu\int_{\mathbb{R}^{N}} |u|^{q}dx+\int_{\mathbb{R}^{N}} |u|^{2_{s}^{\ast}}dx.
\end{equation}
Combining \eqref{intk} with \eqref{intp}, we obtain
$$s||u||^{2}_{D_{s}(\mathbb{R}^{N})}=\mu\gamma_{q,s}s\int_{\mathbb{R}^{N}}|u|^{q}dx+s\int_{\mathbb{R}^{N}}|u|^{2_{s}^{\ast}}dx.$$
As desired.
\end{proof}
Define
 \begin{equation}\label{k}
 (t\star u)(x)=e^{\frac{Nt}{2}}u(e^{t}x)\ \text{for a.e. } x\in\mathbb{R}^{N},
  \end{equation}
it is easy to see that $t\star u\in S_{a}.$ We define the fiber map as follows:
\begin{equation}\label{pp}
 \Psi^{\mu}_{u}(t)=E_{\mu}(t\star u)=\frac{e^{2st}}{2}||u||^{2}_{D_{s}(\mathbb{R}^{N})}-\mu\frac{e^{q\gamma_{q,s}st}}{q}\int_{\mathbb{R}^{N}} |u|^{q}dx-\frac{e^{2_{s}^{\ast}st}}{2_{s}^{\ast}}\int_{\mathbb{R}^{N}} |u|^{2_{s}^{\ast}}dx.
\end{equation}
It is easy to see that $(\Psi^{\mu}_{u})'(t)=P_{\mu}(t\star u)$, so that $t$ is a critical point of $\Psi^{\mu}_{u}(t)$ if and only if $t\star u\in \mathcal{P}_{(a,\mu)}$ and in particular $u\in \mathcal{P}_{(a,\mu)}$ if and only if $0$ is a critical point of $\Psi^{\mu}_{u}(t)$.

We  split $\mathcal{P}_{a,\mu}$ into three parts.
\begin{align*}
\mathcal{P}^{+}_{a,\mu}&=\bigg\{u\in \mathcal{P}_{a,\mu}\mid (\Psi^{\mu}_{u})''(0) >0\bigg\}\\
&=\bigg\{u\in \mathcal{P}_{a,\mu}\mid 2s^{2}||u||^{2}_{D_{s}(\mathbb{R}^{N})} >\mu q\gamma^{2}_{q,s}s^{2}\int_{\mathbb{R}^{N}} |u|^{q}dx+2_{s}^{\ast}s^{2}\int_{\mathbb{R}^{N}} |u|^{2_{s}^{\ast}}dx\bigg\},
\end{align*}
\begin{align*}
\mathcal{P}^{0}_{a,\mu}&=\bigg\{u\in \mathcal{P}_{a,\mu}\mid (\Psi^{\mu}_{u})''(0) =0\bigg\}\\
&=\bigg\{u\in \mathcal{P}_{a,\mu}\mid 2s^{2}||u||^{2}_{D_{s}(\mathbb{R}^{N})} =\mu q\gamma^{2}_{q,s}s^{2}\int_{\mathbb{R}^{N}} |u|^{q}dx+2_{s}^{\ast}s^{2}\int_{\mathbb{R}^{N}} |u|^{2_{s}^{\ast}}dx\bigg\},
\end{align*}
\begin{align}\label{k2}
\mathcal{P}^{-}_{a,\mu}&=\bigg\{u\in \mathcal{P}_{a,\mu}\mid (\Psi^{\mu}_{u})''(0) <0\bigg\}\\\nonumber
&=\bigg\{u\in \mathcal{P}_{a,\mu}\mid 2s^{2}||u||^{2}_{D_{s}(\mathbb{R}^{N})}< \mu q\gamma^{2}_{q,s}s^{2}\int_{\mathbb{R}^{N}} |u|^{q}dx+2_{s}^{\ast}s^{2}\int_{\mathbb{R}^{N}} |u|^{2_{s}^{\ast}}dx\bigg\}.
\end{align}
It is easy to see that $$\mathcal{P}_{a,\mu}=\mathcal{P}^{+}_{a,\mu}\cup \mathcal{P}^{0}_{a,\mu}\cup \mathcal{P}^{-}_{a,\mu}.$$

\begin{lemma}\label{Lem5}
Let $N>2s,\  2<q<2_{s}^{\ast}$ and $a,\mu>0$. Let $\{u_{n}\}\subset S_{a,r}=S_{a}\cap H^{s}(\mathbb{R}^{N})$ be a Palais-Smale  sequence for $E_{\mu}|_{S_{a}}$ at level $m(a,\mu)$. Then $\{u_{n}\}$ is bounded in $H^{s}(\mathbb{R}^{N})$.
\end{lemma}
\begin{proof}
{\bf Case 1:}  $q<\overline{p}$. This yields that $\gamma_{q,s}q<2$. Since $P_{\mu}(u_{n})\rightarrow 0 $, we have  $$s||u||^{2}_{D_{s}(\mathbb{R}^{N})}-\mu\gamma_{q,s}s\int_{\mathbb{R}^{N}} |u|^{q}dx-s\int_{\mathbb{R}^{N}} |u|^{2_{s}^{\ast}}dx=o_{n}(1).$$
Thus, by fractional Gagliardo-Nirenberg-Sobolev inequality \eqref{U3}, we have

 \begin{align*}
E_{\mu}(u_{n})&=\frac{s}{N}||u_{n}||^{2}_{D_{s}(\mathbb{R}^{N})}-\frac{\mu}{q}\left(1-\frac{q\gamma_{q,s}}{2_{s}^{\ast}}\right)\int_{\mathbb{R}^{N}} |u_{n}|^{q}dx+o_{n}(1)\\
&\geq\frac{s}{N}||u_{n}||^{2}_{D_{s}(\mathbb{R}^{N})}-\frac{\mu}{q}\left(1-\frac{q\gamma_{q,s}}{2_{s}^{\ast}}\right)C^{q}_{N,q,s}||u_{n}||^{q\gamma_{q,s}}_{D_{s}(\mathbb{R}^{N})}a^{q(1-\gamma_{q,s})}.
\end{align*}
Since $\{u_{n}\}$ is  a Palais-Smale  sequence for $E_{\mu}|_{S_{a}}$ at level $m(a,\mu)$, we have $E_{\mu}(u_{n})\leq m+1$ for $n$ large. Hence $$\frac{s}{N}||u_{n}||^{2}_{D_{s}(\mathbb{R}^{N})}\leq \frac{\mu}{q}\left(1-\frac{q\gamma_{q,s}}{2_{s}^{\ast}}\right)C^{q}_{N,q,s}||u_{n}||^{q\gamma_{q,s}}_{D_{s}(\mathbb{R}^{N})}a^{q(1-\gamma_{q,s})}+m(a,\mu)+2 ,$$ which implies that $\{u_{n}\}$ is bounded in $H^{s}(\mathbb{R}^{N})$.

\vspace{1mm}

{\bf Case 2:} $q=\overline{p}$. Then $\gamma_{\overline{p},s}\overline{p}=2$. Since  $P_{\mu}(u_{n})\rightarrow 0 $, we know
\begin{align}\label{U4}
||u||^{2}_{D_{s}(\mathbb{R}^{N})}-\mu\gamma_{\overline{p},s}\int_{\mathbb{R}^{N}} |u|^{\overline{p}}dx-\int_{\mathbb{R}^{N}} |u|^{2_{s}^{\ast}}dx=o_{n}(1).
\end{align}
Thus, $$E_{\mu}(u_{n})=\frac{s}{N}\int_{\mathbb{R}^{N}} |u_{n}|^{2_{s}^{\ast}}dx+o_{n}(1)\leq +m(a,\mu)+1\Rightarrow \int_{\mathbb{R}^{N}} |u_{n}|^{2_{s}^{\ast}}dx\leq C.$$
Since $q\in(2,2_{s}^{\ast})$, we have $q=\alpha2+(1-\alpha)2_{s}^{\ast}$ for suitable $\alpha\in (0,1)$, so by H\"{o}lder's inequality, we have $$\int_{\mathbb{R}^{N}} |u_{n}|^{q}dx\leq\left(\int_{\mathbb{R}^{N}} |u_{n}|^{2}dx\right)^{\alpha}\left(\int_{\mathbb{R}^{N}} |u_{n}|^{2_{s}^{\ast}}dx\right)^{1-\alpha}\leq C.$$
Thus, from \eqref{U4}, we know that $$||u||^{2}_{D_{s}(\mathbb{R}^{N})}=\mu\gamma_{\overline{p},s}\int_{\mathbb{R}^{N}} |u|^{\overline{p}}dx+\int_{\mathbb{R}^{N}} |u|^{2_{s}^{\ast}}dx\leq C.$$
{\bf Case 3:} $\overline{p}<q<2_{s}^{\ast}$. This implies that $\gamma_{q,s}q>2$. Since  $P_{\mu}(u_{n})\rightarrow 0 $, we know
\begin{align*}
||u||^{2}_{D_{s}(\mathbb{R}^{N})}-\mu\gamma_{q,s}\int_{\mathbb{R}^{N}} |u|^{q}dx-\int_{\mathbb{R}^{N}} |u|^{2_{s}^{\ast}}dx=o_{n}(1).
\end{align*}
Thus $$E_{\mu}(u_{n})=\frac{\mu}{q}\left(\frac{\gamma_{q,s}q}{2}-1\right)\int_{\mathbb{R}^{N}} |u|^{q}dx+\frac{s}{N}\int_{\mathbb{R}^{N}} |u|^{2_{s}^{\ast}}dx\leq m(a,\mu)+1.$$
So $\int_{\mathbb{R}^{N}} |u|^{q}dx$ and $\int_{\mathbb{R}^{N}} |u|^{2_{s}^{\ast}}dx$ are both bounded. Hence
\begin{align*}
||u||^{2}_{D_{s}(\mathbb{R}^{N})}=\mu\gamma_{q,s}\int_{\mathbb{R}^{N}} |u|^{q}dx+\int_{\mathbb{R}^{N}} |u|^{2_{s}^{\ast}}dx+o_{n}(1)\leq C.
\end{align*}
This completes the proof.
\end{proof}
\begin{proposition}\label{LM5}
Let $N>2s, 2<q<2_{s}^{\ast}$ and $a,\mu>0$. Let $\{u_{n}\}\subset S_{a,r}=S_{a}\bigcap H^{s}(\mathbb{R}^{N})$ be a Palais-Smale  sequence for $E_{\mu}|_{S_{a}}$ at level $m(a,\mu)$ with $$m(a,\mu)<\frac{s}{N}S^{\frac{N}{2s}}_{s}\ \text{and}\ m\neq 0.$$
Suppose in addition that $\mathcal{P}_{a,\mu}(u_{n})\rightarrow 0 \ \text{as} \ n\rightarrow+\infty.$ Then one of the following alternatives holds:
\begin{enumerate}
\item[$(i)$]  either up to a subsequence $u_{n}\rightharpoonup u$ weakly in $H^{s}(\mathbb{R}^{N})$ but not strongly , where $u\not\equiv0$ is a solution of \eqref{int1} for some $\lambda<0$, and $$E_{\mu}(u)\leq m(a,\mu)- \frac{s}{N}S^{\frac{N}{2s}}_{s}.$$
\item[$(ii)$] or up to a subsequence $u_{n}\rightarrow u$ strongly in $H^{s}(\mathbb{R}^{N}),$ $E_{\mu}(u)=m(a,\mu)$ and $u$ solves \eqref{int1}--\eqref{in2} for some $\lambda<0.$
\end{enumerate}
\end{proposition}
\begin{proof}
By Lemma \ref{Lem5}, we know that the sequence $\{u_{n}\}$ is bounded in $H^{s}(\mathbb{R}^{N})$, which is radial functions, and by compactness of $H_{rad}^{s}(\mathbb{R}^{N})\hookrightarrow\hookrightarrow L^{q}(\mathbb{R}^{N})$, which implies that $$u_{n}\rightharpoonup u \  \text{in}\  H^{s}(\mathbb{R}^{N}),\ \ u_{n}\rightarrow  u \  \text{in}\ L^{q}(\mathbb{R}^{N})  \ a.e\  \text{in}\ \mathbb{R}^{N} .$$ Since $\{u_{n}\}$ is a bounded Palais-Smale  sequence for $E_{\mu}|_{S_{a}}$ at level $m(a,\mu)$, by Lagrange multipliers rule, there exists $\{\lambda_{n}\}\subset \mathbb{R}$  such that for every $\varphi \in H^{s}(\mathbb{R}^{N}) $
\begin{align}\label{U5}
\int_{\mathbb{R}^{N}} \left((-\Delta)^{\frac{s}{2}}u (-\Delta)^{\frac{s}{2}}\varphi dx-\lambda_{n}u_{n}\varphi-\mu|u|^{q-2}u_{n}\varphi-|u|^{2_{s}^{\ast}-2}u_{n}\varphi\right)dx=o_{n}(1)\|\varphi\|, \ \text{as}\ n\rightarrow+\infty.
\end{align}
If we choose that $\varphi=u_{n}$, from \eqref{U5}, it is easy to see that $\{{u_{n}}\}$ is bounded, hence up to a subsequence $\lambda_{n}\rightarrow \lambda\in \mathbb{R} $. By the fact that $P_{\mu}(u_{n})\rightarrow 0 $ and $\gamma_{q,s}<1$, we deduce that
\begin{align}\label{U6}
\lambda a^{2}&=\lim_{n\rightarrow +\infty}\lambda_{n}\int_{\mathbb{R}^{N}} |u_{n}|^{2}dx=\lim_{n\rightarrow +\infty}\left(||u_{n}||^{2}_{D_{s}(\mathbb{R}^{N})}-\int_{\mathbb{R}^{N}}\left(\mu|u_{n}|^{q}+|u_{n}|^{2_{s}^{\ast}}\right)dx\right)\\\nonumber
&=\lim_{n\rightarrow +\infty}\mu(\gamma_{q,s}-1)\int_{\mathbb{R}^{N}}|u_{n}|^{q}dx=\mu(\gamma_{q,s}-1)\int_{\mathbb{R}^{N}}|u|^{q}dx\leq0.
\end{align}
It is easy to see that $\lambda=0$ if and only if $u\equiv 0$. Next, we show that the $u\not\equiv0.$
Assume by contradiction that $u\equiv0$, by $\{u_{n}\}$ is bounded in $H^{s}(\mathbb{R}^{N})$, hence up to a subsequence $||u_{n}||^{2}_{D_{s}(\mathbb{R}^{N})}\rightarrow \ell\in\mathbb{R}$. From $P_{\mu}(u_{n})\rightarrow0$ and $u_{n}\rightarrow 0$ strongly in $L^{q}(\mathbb{R}^{N})$, hence
\begin{align*}
\int_{\mathbb{R}^{N}} |u|^{2_{s}^{\ast}}dx=||u||^{2}_{D_{s}(\mathbb{R}^{N})}-\mu\gamma_{q,s}\int_{\mathbb{R}^{N}} |u|^{q}dx\rightarrow \ell ,
\end{align*}
Therefore, by the definition of $S_{s}$ in \eqref{int2}, we have $\ell\geq S_{s}\ell^{\frac{2}{2_{s}^{\ast}}}$, we can deduce $$\ell=0\ \text{ or}\  \ell\geq S^{\frac{N}{2s}}_{s}.$$
{\bf Case 1.} If $||u_{n}||^{2}_{D_{s}(\mathbb{R}^{N})}\rightarrow \ell=0$, then $\int_{\mathbb{R}^{N}}|u_{n}|^{q}dx\rightarrow0,\ \ \int_{\mathbb{R}^{N}}|u_{n}|^{2_{s}^{\ast}}dx\rightarrow0 $, which implies that $E_{\mu}(u_{n})\rightarrow 0$, this contradict the fact that $E_{\mu}(u_{n})\rightarrow m(a,\mu)$.\\
{\bf Case 2.}  If $\ell\geq S^{\frac{N}{2s}}_{s}$, from $P_{\mu}(u_{n})\rightarrow 0$ and $E_{\mu}(u_{n})\rightarrow m(a,\mu)$, we obtain
\begin{align*}
m(a,\mu)+o_{n}(1)&=E_{\mu}(u_{n})=\frac{s}{N}||u_{n}||^{2}_{D_{s}(\mathbb{R}^{N})}-\frac{\mu}{q}\left(1-\frac{q\gamma_{q,s}}{2_{s}^{\ast}}\right)\int_{\mathbb{R}^{N}} |u|^{q}dx+o_{n}(1)\\
&=\frac{s}{N}||u_{n}||^{2}_{D_{s}(\mathbb{R}^{N})}+o_{n}(1)= \frac{s}{N}\ell+o_{n}(1) ,
\end{align*}
which implies that $$m(a,\mu)=\frac{s}{N}\ell\geq \frac{s}{N}S^{\frac{N}{2s}}_{s},$$ which contradicts our assumptions. Thus, $u\not\equiv0.$ From \eqref{U6}, we know that $\lambda<0$. Pass to the limit in \eqref{U5} by the weak convergence, we obtain that
\begin{align}\label{U8}
(-\Delta)^{s}u=\lambda u +\mu|u|^{q-2}u+|u|^{2_{s}^{\ast}-2}u,\ &x\in\mathbb{R}^{N}.
\end{align}
By the Pohozaev identity, $P_{\mu}(u)=0.$ Let $\sigma_{n}=u_{n}-u$, then $\sigma_{n}\rightharpoonup 0$ in $H^{s}(\mathbb{R}^{N})$. Since
\begin{align}\label{U9}
||u_{n}||^{2}_{D_{s}(\mathbb{R}^{N})}=||\sigma_{n}||^{2}_{D_{s}(\mathbb{R}^{N})}+||u||^{2}_{D_{s}(\mathbb{R}^{N})}+o_{n}(1)
\end{align}
and by the well-known Br\'ezis-Lieb lemma, we get
\begin{align}\label{U10}
\int_{\mathbb{R}^{N}} |u_{n}|^{2_{s}^{\ast}}dx=\int_{\mathbb{R}^{N}} |\sigma_{n}|^{2_{s}^{\ast}}dx+\int_{\mathbb{R}^{N}} |u|^{2_{s}^{\ast}}dx+o_{n}(1)
\end{align}
Therefore, from $P_{\mu}(u_{n})\rightarrow 0$ and $u_{n}\rightarrow u$ in $L^{q}(\mathbb{R}^{N})$, we have $$||\sigma_{n}||^{2}_{D_{s}(\mathbb{R}^{N})}+||u||^{2}_{D_{s}(\mathbb{R}^{N})}=\mu\gamma_{q,s}\int_{\mathbb{R}^{N}} |u|^{q}dx+\int_{\mathbb{R}^{N}} |\sigma_{n}|^{2_{s}^{\ast}}dx+\int_{\mathbb{R}^{N}} |u|^{2_{s}^{\ast}}dx+o_{n}(1).$$Combining this with $P_{\mu}(u)=0,$ we know that $||\sigma_{n}||^{2}_{D_{s}(\mathbb{R}^{N})}=\int_{\mathbb{R}^{N}} |\sigma_{n}|^{2_{s}^{\ast}}dx+o_{n}(1),$ thus $$\lim_{n\rightarrow +\infty}||\sigma_{n}||^{2}_{D_{s}(\mathbb{R}^{N})}=\lim_{n\rightarrow +\infty}\int_{\mathbb{R}^{N}} |\sigma_{n}|^{2_{s}^{\ast}}dx=\ell\geq0.$$By the definition of $S_{s}$ in \eqref{int2}, we have $\ell\geq S_{s}\ell^{\frac{2}{2_{s}^{\ast}}}$, hence we can deduce $$\ell=0\ \text{ or}\  \ell\geq S^{\frac{N}{2s}}_{s}.$$
{\bf Case 1.} $\ell\geq S^{\frac{N}{2s}}_{s}$. By \eqref{U9} and \eqref{U10}, we have
\begin{align*}
m(a,\mu)&=\lim_{n\rightarrow +\infty}E_{\mu}(u_{n})=\lim_{n\rightarrow +\infty}\left(E_{\mu}(u)+\frac{1}{2}||\sigma_{n}||^{2}_{D_{s}(\mathbb{R}^{N})}-\frac{1}{2_{s}^{\ast}}\int_{\mathbb{R}^{N}} |\sigma_{n}|^{2_{s}^{\ast}}dx\right)\\
&=E_{\mu}(u)+\frac{s}{N}\ell\geq E_{\mu}(u)+\frac{s}{N}S^{\frac{N}{2s}}_{s}.
\end{align*}
Thus, the conclusion i) holds, i.e., up to a subsequence, $u_{n}\rightharpoonup u$ weakly in $H^{s}(\mathbb{R}^{N})$ but not strongly, where $u\not\equiv0$ is a solution of \eqref{int1} for some $\lambda<0$, and $$E_{\mu}(u)\leq m(a,\mu)- \frac{s}{N}S^{\frac{N}{2s}}_{s}.$$
{\bf Case 2.} $\ell=0.$ Then $u_{n}\rightarrow u$ strongly in $D_{s}(\mathbb{R}^{N})$, which implies that $u_{n}\rightarrow u$ strongly in $L^{2_{s}^{\ast}}(\mathbb{R}^{N})$ by Sobolev embedding inequality.  Next, we show that $u_{n}\rightarrow u$ strongly in $L^{2}(\mathbb{R}^{N}).$ If we let $\varphi=u_{n}-u$ in \eqref{U5} and multiply $u_{n}-u$ on both side of \eqref{U8}, we obtain
\begin{align*}
&||u_{n}-u||^{2}_{D_{s}(\mathbb{R}^{N})}-\int_{\mathbb{R}^{N}}(\lambda_{n}u_{n}-\lambda u)(u_{n}-u)dx\\
&=\int_{\mathbb{R}^{N}}(|u_{n}|^{q-2}u_{n}-|u|^{q-2}u)(u_{n}-u)dx+\int_{\mathbb{R}^{N}}(|u_{n}|^{2_{s}^{\ast}-2}u_{n}-|u|^{2_{s}^{\ast}-2}u)(u_{n}-u)dx+o_{n}(1)
\end{align*}
Thus, by $u_{n}\rightarrow u$ strongly in $D_{s}(\mathbb{R}^{N})$ and $u_{n}\rightarrow u$ strongly in $L^{2_{s}^{\ast}}(\mathbb{R}^{N})$, we have $$0=\lim_{n\rightarrow +\infty}\int_{\mathbb{R}^{N}}(\lambda_{n}u_{n}-\lambda u)(u_{n}-u)dx=\lim_{n\rightarrow +\infty}\lambda\int_{\mathbb{R}^{N}}(u_{n}-u)^{2}dx,$$
which implies that $u_{n}\rightarrow u$ strongly in $L^{2}(\mathbb{R}^{N})$ by $\lambda<0$. Thus, the conclusion ii) holds, i.e. up to a subsequence $u_{n}\rightarrow u$ strongly in $H^{s}(\mathbb{R}^{N}),$ $E_{\mu}(u)=m(a,\mu)$ and $u$ solves \eqref{int1}--\eqref{in2} for some $\lambda<0.$
The proof is thus complete.
\end{proof}
By the similar arguments as in Proposition \ref{LM5}, we can obtain the following proposition.
\begin{proposition}\label{LM6}
Let $N>2s, 2<q<2_{s}^{\ast}$ and $a,\mu>0$. Let $\{u_{n}\}\subset S_{a}$ be a Palais-Smale  sequence for $E_{\mu}|_{S_{a}}$ at level $m(a,\mu)$ with $$m(a,\mu)<\frac{s}{N}S^{\frac{N}{2s}}_{s}\ \text{and}\ m\neq 0.$$Suppose in addition that $\mathcal{P}_{a,\mu}(u_{n})\rightarrow 0 \ \text{as} \ n\rightarrow+\infty,$ and that there exists $\{v_{n}\}\subset S_{a}$ and $v_{n}$ is radially symmetric for every $n$ such that $\|u_{n}-v_{n}\|\rightarrow 0$ as $n\rightarrow+\infty.$ Then one of the alternatives (i) and (ii) in Proposition \ref{LM5} holds.
\end{proposition}
\section{$L^{2}$-subcritical perturbation}\label{sec2}
For $N>2s$  and $2<q<2+4s/N$, let us recall $C'$ in \eqref{equ2}.
We consider the constrained functional $E_{\mu}|_{S_{a}}$. For every $u\in S_{a}$, by fractional Gagliardo-Nirenberg-Sobolev inequality \eqref{U3} and Sobolev inequality \eqref{int2}
\begin{align}\label{U16}
E_{\mu}(u)\geq\frac{1}{2}||u||^{2}_{D_{s}(\mathbb{R}^{N})}
-\frac{\mu}{q}C^{q}_{N,q,s}||u||^{q\gamma_{q,s}}_{D_{s}(\mathbb{R}^{N})}a^{q(1-\gamma_{p,s})}-\frac{1}{2_{s}^{\ast}}S^{-\frac{2_{s}^{\ast}}{2}}_{s}||u||^{2_{s}^{\ast}}_{D_{s}(\mathbb{R}^{N})}.
\end{align}
Therefore, we consider the function $h:\mathbb{R}^{+}\rightarrow \mathbb{R}$
\begin{align}\label{equ5}
h(t)=\frac{1}{2}t^{2}-\frac{\mu}{q}C^{q}_{N,q,s}a^{q(1-\gamma_{q,s})}t^{q\gamma_{q,s}}-\frac{1}{2_{s}^{\ast}}S^{-\frac{2_{s}^{\ast}}{2}}_{s}t^{2_{s}^{\ast}}.
\end{align}
Since $\mu>0$ and $q\gamma_{q,s}<2<2_{s}^{\ast}$, we have $h(0^{+})=0^{-}$ and $h(+\infty)=-\infty.$
\begin{lemma}\label{Lem2}
 Under the assumption that $\mu a^{(1-\gamma_{q,s})q}< C'$ (see \eqref{equ2}), the function $h$ has a local strict minimum at negative level, a global strict maximum at positive level, and no other critical points, and there exists a $R_{0}$ and $R_{1}$ both depending on $a$ and $\mu$, such that $h(R_{0})=0=h(R_{1})$ and $h(t)\geq0$ if and only if $t\in (R_{0}, R_{1}).$
\end{lemma}
\begin{proof}
For $t>0$, we have $h(t)>0$ if and only if $$\varphi(t)>\frac{\mu}{q}C^{q}_{N,q,s}a^{q(1-\gamma_{q,s})},\ \text{with}\ \varphi(t)=\frac{1}{2}t^{2-q\gamma_{q,s}}-\frac{1}{2_{s}^{\ast}}S^{-\frac{2_{s}^{\ast}}{2}}_{s}t^{2_{s}^{\ast}-q\gamma_{q,s}} .$$ Since $$\varphi'(t)=\frac{2-q\gamma_{q,s}}{2}t^{1-q\gamma_{q,s}}-\frac{2_{s}^{\ast}-q\gamma_{q,s}}{2_{s}^{\ast}}S^{-\frac{2_{s}^{\ast}}{2}}_{s}t^{2_{s}^{\ast}-1-q\gamma_{q,s}},$$
it is easy to see that $\varphi(t)$ is increasing on $(0, \overline{t})$ and decreasing on $(\overline{t},+\infty)$ and has a unique global maximum point at positive level on $(0,+\infty)$, where $\overline{t}=\left(\frac{(2-q\gamma_{q,s})2_{s}^{\ast}}{2(2_{s}^{\ast}-q\gamma_{q,s})}S^{\frac{2_{s}^{\ast}}{2}}_{s}\right)^{\frac{1}{2_{s}^{\ast}-2}}$. Thus the maximum level is $$\varphi(\overline{t})=\frac{2_{s}^{\ast}-2}{2(2_{s}^{\ast}-q\gamma_{q,s})}\left(\frac{(2-q\gamma_{q,s})2_{s}^{\ast}}{2(2_{s}^{\ast}-q\gamma_{q,s})}S^{\frac{2_{s}^{\ast}}{2}}_{s}\right)^{\frac{2-q\gamma_{q,s}}{2_{s}^{\ast}-2}}.$$
Therefore, $h$ is positive on an open interval $(R_{0},R_{1})$ if and only if $\varphi(\overline{t})>\frac{\mu}{q}C^{q}_{N,q,s}a^{q(1-\gamma_{q,s})}$, which implies that $$\mu a^{q(1-\gamma_{q,s})} < \frac{q(2_{s}^{\ast}-2)}{2C^{q}_{N,q,s}(2_{s}^{\ast}-q\gamma_{q,s})}\left(\frac{(2-q\gamma_{q,s})2_{s}^{\ast}}{2(2_{s}^{\ast}-q\gamma_{q,s})}S^{\frac{2_{s}^{\ast}}{2}}_{s}\right)^{\frac{2-q\gamma_{q,s}}{2_{s}^{\ast}-2}}.$$
Since $h(0^{+})=0^{-}$ ,  $h(+\infty)=-\infty$  and $h$ is positive on an open interval $(R_{0},R_{1})$ , it is easy to see that $h$ has a global maximum at positive level in $(R_{0},R_{1})$ and has a local minimum point at negative level in $(0,R_{0})$. Since $h'(t)=t^{q\gamma_{q,s}-1}\left[t^{2-q\gamma_{q,s}}-\mu\gamma_{q,s}C^{q}_{N,q,s}a^{q(1-\gamma_{q,s})}-S^{-\frac{2_{s}^{\ast}}{2}}_{s}t^{2_{s}^{\ast}-q\gamma_{q,s}}\right]=0$ if and only if $$\psi(t)=\mu\gamma_{q,s}C^{q}_{N,q,s}a^{q(1-\gamma_{q,s})}\ \text{where}\ \psi(t)=t^{2-q\gamma_{q,s}}-S^{-\frac{2_{s}^{\ast}}{2}}_{s}t^{2_{s}^{\ast}-q\gamma_{q,s}}.$$ Obviously,
$\psi(t)$ has only one critical point, which is a strict maximum.
Therefore, if $\psi(t)_{max}\leq\mu\gamma_{q,s}C^{q}_{N,q,s}a^{q(1-\gamma_{q,s})},$
then it is easy to see that contract with $h$ is positive on an open interval $(R_{0},R_{1})$. Thus, $\psi(t)_{max}>\mu\gamma_{q,s}C^{q}_{N,q,s}a^{q(1-\gamma_{q,s})}$, which implies that $h$ only has a local strict minimum at negative level,
a global strict maximum at positive level, and no other critical points,
\end{proof}
\begin{lemma}\label{Lem3}
Under the condition of $\mu a^{(1-\gamma_{q,s})q}< C'$ (see \eqref{equ2}), then $\mathcal{P}^{0}_{a,\mu}=\emptyset$ and $\mathcal{P}_{a,\mu}$ is a smooth manifold of codimension 1 in $S_{a}$.
\end{lemma}
\begin{proof}
Assume by contradiction that there exists a $u\in \mathcal{P}^{0}_{a,\mu}$ such that
\begin{align}\label{U11}
||u||^{2}_{D_{s}(\mathbb{R}^{N})}-\mu\gamma_{q,s}\int_{\mathbb{R}^{N}} |u|^{q}dx-\int_{\mathbb{R}^{N}} |u|^{2_{s}^{\ast}}dx=0,
\end{align}
and
\begin{align}\label{U12}
2||u||^{2}_{D_{s}(\mathbb{R}^{N})} =\mu q\gamma^{2}_{q,s}\int_{\mathbb{R}^{N}} |u|^{q}dx+2_{s}^{\ast}\int_{\mathbb{R}^{N}} |u|^{2_{s}^{\ast}}dx.
\end{align}
Thus, from \eqref{int2}, \eqref{U3}, \eqref{U11}, and \eqref{U12}, we have $$\mu\gamma_{q,s}(2-q\gamma_{q,s})\int_{\mathbb{R}^{N}} |u|^{q}dx=(2_{s}^{\ast}-2)\int_{\mathbb{R}^{N}} |u|^{2_{s}^{\ast}}dx,$$
\begin{align}\label{U13}
||u||^{2}_{D_{s}(\mathbb{R}^{N})} =\frac{2_{s}^{\ast}-q\gamma_{q,s}}{2-q\gamma_{q,s}}\int_{\mathbb{R}^{N}} |u|^{2_{s}^{\ast}}dx\leq \frac{2_{s}^{\ast}-q\gamma_{q,s}}{2-q\gamma_{q,s}} S^{-\frac{2_{s}^{\ast}}{2}}_{s}||u||^{2_{s}^{\ast}}_{D_{s}(\mathbb{R}^{N})},
\end{align}
\begin{align}\label{U14}
||u||^{2}_{D_{s}(\mathbb{R}^{N})} =\mu\gamma_{q,s}\frac{2_{s}^{\ast}-q\gamma_{q,s}}{2_{s}^{\ast}-2}\int_{\mathbb{R}^{N}} |u|^{q}dx
\leq \mu\gamma_{q,s}\frac{2_{s}^{\ast}-q\gamma_{q,s}}{2_{s}^{\ast}-2} C^{q}_{N,q,s}a^{q(1-\gamma_{q,s})}||u||^{q\gamma_{q,s}}_{D_{s}(\mathbb{R}^{N})},
\end{align}
From \eqref{U13} and \eqref{U14}, we can infer that $$\left[\frac{2-q\gamma_{q,s}}{2_{s}^{\ast}-q\gamma_{q,s}} S^{\frac{2_{s}^{\ast}}{2}}_{s}\right]^{\frac{1}{2_{s}^{\ast}-2}}\leq \left[\mu\gamma_{q,s}\frac{2_{s}^{\ast}-q\gamma_{q,s}}{2_{s}^{\ast}-2} C^{q}_{N,q,s}a^{q(1-\gamma_{q,s})}\right]^{\frac{1}{2-q\gamma_{q,s}}},$$
which implies that
\begin{align}\label{U15}
\mu a^{q(1-\gamma_{q,s})}\geq \frac{2_{s}^{\ast}-2}{\gamma_{q,s}C^{q}_{N,q,s}(2_{s}^{\ast}-q\gamma_{q,s})} \left[\frac{2-q\gamma_{q,s}}{2_{s}^{\ast}-q\gamma_{q,s}} S^{\frac{2_{s}^{\ast}}{2}}_{s}\right]^{\frac{2-q\gamma_{q,s}}{2_{s}^{\ast}-2}}.
\end{align}
Next, we show that the right hand of \eqref{U15} is greater than or equal to $C'$. To show $$\frac{2_{s}^{\ast}-2}{\gamma_{q,s}C^{q}_{N,q,s}(2_{s}^{\ast}-q\gamma_{q,s})} \left[\frac{2-q\gamma_{q,s}}{2_{s}^{\ast}-q\gamma_{q,s}} S^{\frac{2_{s}^{\ast}}{2}}_{s}\right]^{\frac{2-q\gamma_{q,s}}{2_{s}^{\ast}-2}}\geq \frac{q(2_{s}^{\ast}-2)}{2C^{q}_{N,q,s}(2_{s}^{\ast}-q\gamma_{q,s})}\left(\frac{(2-q\gamma_{q,s})2_{s}^{\ast}}{2(2_{s}^{\ast}-q\gamma_{q,s})}S^{\frac{2_{s}^{\ast}}{2}}_{s}\right)^{\frac{2-q\gamma_{q,s}}{2_{s}^{\ast}-2}},$$
we only need to prove that $$\left(\frac{q\gamma_{q,s}}{2}\right)^{2_{s}^{\ast}-2}\left(\frac{2_{s}^{\ast}}{2}\right)^{2-q\gamma_{q,s}}\leq1, \ \text{for every} \ \ 2<q<\overline{p}<2_{s}^{\ast}.$$
Let $q\gamma_{q,s}=x\in (0,2)$, we need to show that $\left(\frac{x}{2}\right)^{2_{s}^{\ast}-2}\left(\frac{2_{s}^{\ast}}{2}\right)^{2-x}\leq1.$ For this, we set $f(x)=\left(\frac{x}{2}\right)^{2_{s}^{\ast}-2}\left(\frac{2_{s}^{\ast}}{2}\right)^{2-x}$, it is easy to see that $f(x)$ is strictly increasing on $(0, \frac{2_{s}^{\ast}-2}{\ln2_{s}^{\ast}-\ln 2})$ and decreasing on $( \frac{2_{s}^{\ast}-2}{\ln2_{s}^{\ast}-\ln 2},+\infty)$. Thus, when $x\in (0,2)$ $f(x)\leq f(2)=1,$ which implies that $$\left(\frac{q\gamma_{q,s}}{2}\right)^{2_{s}^{\ast}-2}\left(\frac{2_{s}^{\ast}}{2}\right)^{2-q\gamma_{q,s}}\leq1.$$ This contradicts the assumption $\mu a^{q(1-\gamma_{q,s})}<C'.$  Thus, $\mathcal{P}^{0}_{a,\mu}=\emptyset.$

Next, we show that $\mathcal{P}_{a,\mu}$ is a smooth manifold of codimension 1 on $S_{a}$. Since $\mathcal{P}_{a,\mu}=\bigg\{u\in S_{a}:||u||^{2}_{D_{s}(\mathbb{R}^{N})}=\mu\gamma_{q,s}\int_{\mathbb{R}^{N}} |u|^{q}dx+\int_{\mathbb{R}^{N}} |u|^{2_{s}^{\ast}}dx\bigg\}$, we know that $\mathcal{P}_{a,\mu}$ is defined by $P_{\mu}(u)=0$, $G(u)=0$, where $$P_{\mu}(u)=s||u||^{2}_{D_{s}(\mathbb{R}^{N})}-\mu\gamma_{q,s}s\int_{\mathbb{R}^{N}} |u|^{q}dx-s\int_{\mathbb{R}^{N}} |u|^{2_{s}^{\ast}}dx \ \ \text{and}\ \ G(u)=\int_{\mathbb{R}^{N}} |u|^{2}dx=a^{2} .$$ Since $P_{\mu}(u)$ and $G(u)$ are class of $C^{1}$, we only need to check that $d(P_{\mu}(u),G(u))$: $H^{s}(\mathbb{R}^{N})\rightarrow \mathbb{R}^{2}$ is surjective. If this not true, $dP_{\mu}(u)$ has to be linearly dependent from $dG(u)$ i.e. there exist a $\nu\in \mathbb{R}$ such that $$2s\int_{\mathbb{R}^{N}} (-\Delta)^{\frac{s}{2}}u(-\Delta)^{\frac{s}{2}}\varphi dx-\mu q\gamma_{q,s}s\int_{\mathbb{R}^{N}} |u|^{q-2}u\varphi dx-s2_{s}^{\ast}\int_{\mathbb{R}^{N}} |u|^{2_{s}^{\ast}-2}u\varphi dx=\nu\int_{\mathbb{R}^{N}} u\varphi dx$$ for every $\varphi\in H^{s}(\mathbb{R}^{N})$, which implies that $$2s(-\Delta)^{2}u=\nu u+\mu q\gamma_{q,s}su^{q-1}+2_{s}^{\ast}s u^{2_{s}^{\ast}-1}\ \ \text{in}\ \mathbb{R}^{N}.$$
By the Pohozaev identity for above equation, we know that $$2s^{2}||u||^{2}_{D_{s}(\mathbb{R}^{N})}= \mu q\gamma^{2}_{q,s}s^{2}\int_{\mathbb{R}^{N}} |u|^{q}dx+2_{s}^{\ast}s^{2}\int_{\mathbb{R}^{N}} |u|^{2_{s}^{\ast}}dx,$$
that is $u\in \mathcal{P}^{0}_{a,\mu}$, a contradiction. Hence, $\mathcal{P}_{a,\mu}$ is a natural constraint.
\end{proof}
\begin{lemma}\label{Lem4}
For every $u\in S_{a}$, the function $\Psi^{\mu}_{u}(t)$ has exactly two critical points $a_{u}<t_{u}\in \mathbb{R}$ and two zeros $c_{u}<d_{u}\in \mathbb{R}$, with $a_{u}<c_{u}<t_{u}<d_{u}$. Moreover,
\begin{enumerate}
\item[$(1)$] $a_{u}\star u\in \mathcal{P}^{+}_{a,\mu}$ and $t_{u}\star u\in \mathcal{P}^{-}_{a,\mu}$, and if $t\star u\in \mathcal{P}_{a,\mu}$, then either $t=a_{u}$ or $t=t_{u}.$

\item[$(2)$] $||t\star u||_{D_{s}(\mathbb{R}^{N})}\leq R_{0}$ for every $t\leq c_{u},$ and $$E_{\mu}(u)(a_{u}\star u)=\min\{E_{\mu}(t\star u):t\in \mathbb{R}\ \text{and}\ ||t\star u||_{D_{s}(\mathbb{R}^{N})}< R_{0} \}<0.$$

\item[$(3)$] We have $$E_{\mu}(u)(t_{u}\star u)=\max\{E_{\mu}(t\star u):t\in \mathbb{R} \}>0$$ and $\Psi^{\mu}_{u}(t)$ is strictly decreasing and concave on $(t_{u},+\infty)$. In particular, if $t_{u}<0$, then $P_{\mu}(u)<0.$

\item[$(4)$] The maps $u\in S_{a}: a_{u} \in \mathbb{R} $ and $u\in S_{a}: t_{u} \in \mathbb{R} $ are of class $C^{1}$.
\end{enumerate}
\end{lemma}
\begin{proof}
Let $u\in S_{a}$, since $t\star u\in \mathcal{P}_{a,\mu} $  if and only if $(\Psi^{\mu}_{u})'(t)=0$. Thus, we first show that $\Psi^{\mu}_{u}(t)$ has at least two critical points. From \eqref{U16}, we have $$\Psi^{\mu}_{u}(t)=E_{\mu}(t\star u)\geq h(||t\star u||_{D_{s}(\mathbb{R}^{N})})=h(e^{st}|| u||_{D_{s}(\mathbb{R}^{N})}).$$ Thus, the $C^{2}$ function $\Psi^{\mu}_{u}(t)$ is positive on $\left(\frac{\ln\left(\frac{R_{0}}{|| u||_{D_{s}(\mathbb{R}^{N})}}\right)}{s}, \frac{\ln\left(\frac{R_{1}}{|| u||_{D_{s}(\mathbb{R}^{N})}}\right)}{s}\right)$ and $\Psi^{\mu}_{u}(-\infty)=0^{-}$, $\Psi^{\mu}_{u}(+\infty)=-\infty$, thus it is easy to see that $\Psi^{\mu}_{u}(t)$ has a local minimum point $a_{u}$ at negative level in $(0,\frac{\ln\left(\frac{R_{0}}{|| u||_{D_{s}(\mathbb{R}^{N})}}\right)}{s})$ and has a global maximum point $t_{u}$  at positive level in $\left(\frac{\ln\left(\frac{R_{0}}{|| u||_{D_{s}(\mathbb{R}^{N})}}\right)}{s}, \frac{\ln\left(\frac{R_{1}}{|| u||_{D_{s}(\mathbb{R}^{N})}}\right)}{s}\right).$ Next, we show that $\Psi^{\mu}_{u}(t)$ has no other critical points. Indeed, $(\Psi^{\mu}_{u})'(t)=0$ implies that $$\Psi(t)=\mu \gamma_{q,s}s\int_{\mathbb{R}^{N}} |u|^{q}dx\ \text{where}\ \Psi(t)=se^{(2-q\gamma_{q,s})st}\|u\|^{2}_{D_{s}(\mathbb{R}^{N})}-se^{(2_{s}^{\ast}-q\gamma_{q,s})st}\int_{\mathbb{R}^{N}} |u|^{2_{s}^{\ast}}dx.$$ It is easy to see that $\Psi(t)$ has a unique maximum point, thus the above equation has at most two solutions. From $u\in S_{a},$ $t\in\mathbb{R}$ is a critical point of $\Psi^{\mu}_{u}(t)$ if and only if $t\star u\in \mathcal{P}_{a,\mu} ,$ we have $a_{u}\star u,\ \  t_{u}\star u\in \mathcal{P}_{a,\mu} $ and $t\star u \in \mathcal{P}_{a,\mu}$ if and only if $t=a_{u}$ or $t=t_{u}$. Since $a_{u}$ is a local minimum point of $\Psi^{\mu}_{u}(t)$, we know that $(\Psi^{\mu}_{a_{u}\star u})''(0)=(\Psi^{\mu}_{u})''(a_{u})\geq0$, since $\mathcal{P}^{0}_{a,\mu}=\emptyset$, we know that $(\Psi^{\mu}_{u})''(a_{u})\neq0$, thus $(\Psi^{\mu}_{a_{u}\star u})''(0)=(\Psi^{\mu}_{u})''(a_{u})>0$, which implies that $a_{u}\star u\in \mathcal{P}^{+}_{a,\mu},$ similarly, we have $t_{u}\star u\in \mathcal{P}^{-}_{a,\mu}.$ By the monotonicity and the behavior at infinity of $\Psi^{\mu}_{u}$, we know that $\Psi^{\mu}_{u}$ has exactly two zeros $c_{u}<d_{u}$ with $a_{u}<c_{u}<t_{u}<d_{u}$ and $\Psi^{\mu}_{u}$ has exactly two inflection points, in particular, $\Psi^{\mu}_{u}$ is concave on $[t_{u},+\infty)$ and hence if $t_{u}<0$, then $P_{\mu}(u)=(\Psi^{\mu}_{u})'(0)<0.$ Finally, we prove that $u\in S_{a}: a_{u} \in \mathbb{R} $ and $u\in S_{a}: t_{u} \in \mathbb{R} $ are of class $C^{1}$. Indeed, we can apply the implicit function theorem on the $C^{1}$ function $\Phi(t,u)=(\Psi^{\mu}_{u})'(t)$, then $\Phi(a_{u},u)=(\Psi^{\mu}_{u})'(a_{u})=0,  \partial_{s}\Phi(a_{u},u)=(\Psi^{\mu}_{u})''(a_{u})<0$, by the implicit function theorem, we know that $u\in S_{a}: a_{u} \in \mathbb{R} $  is class of $C^{1}$, similarly, we can prove that $u\in S_{a}: t_{u} \in \mathbb{R} $  is class of $C^{1}$.
\end{proof}
For $k>0$,  set $$A_{k}=\{u\in S_{a}:\|u\|^{2}_{D_{s}(\mathbb{R}^{N})}<k\}, \ \text{and}\ m(a,\mu)=\inf_{u\in A_{R_{0}}}E_{\mu}(u).$$
From Lemma \ref{Lem4}, we immediately have the following corollary:
\begin{corollary}\label{Lem6}
The set $\mathcal{P}^{+}_{a,\mu}$ is contained in $$A_{R_{0}}=\{u\in S_{a}: || u||_{D_{s}(\mathbb{R}^{N})}< R_{0}\}\ \text{ and }\ \sup_{\mathcal{P}^{+}_{a,\mu}}E_{\mu}\leq0\leq \inf_{\mathcal{P}^{-}_{a,\mu}}E_{\mu}.$$
\end{corollary}
\begin{lemma}\label{Lem7}
We have $m(a,\mu)\in (-\infty,0)$ that $$m(a,\mu)=\inf_{\mathcal{P}_{a,\mu}}E_{\mu}=\inf_{\mathcal{P}^{+}_{a,\mu}}E_{\mu}\ \text{and that}\ m(a,\mu)< \inf_{\overline{A_{R_{0}}}\setminus A_{R_{0}-\rho}}E_{\mu}$$ for $\rho>0$ small enough.
\end{lemma}
\begin{proof}
For $u\in A_{R_{0}}$, we have $$E_{\mu}(u)\geq h(\|u\|_{D_{s}(\mathbb{R}^{N})})\geq \min_{t\in [0,R_{0}]}h(t)>-\infty.$$ Therefore, $m(a,\mu)>-\infty.$ Moreover, for any $u\in S_{a}$, we obtain $|| t\star u||_{D_{s}(\mathbb{R}^{N})}< R_{0}$ and $E_{\mu}(t\star u)<0$ for $t\ll -1$ and hence $m(a,\mu)<0.$ Since $\mathcal{P}^{+}_{a,\mu}\subset A_{R_{0}}$, we know that $m(a,\mu)\leq \inf_{\mathcal{P}^{+}_{a,\mu}}E_{\mu} .$ On the other hand, if $u\in A_{R_{0}} $, then $a_{u}\star u \in \mathcal{P}^{+}_{a,\mu}\subset A_{R_{0}}$ and $$E_{\mu}(u)(a_{u}\star u)=\min\{E_{\mu}(t\star u):t\in \mathbb{R}\ \text{and}\ ||t\star u||_{D_{s}(\mathbb{R}^{N})}< R_{0} \}\leq E_{\mu}(u),$$ which implies that $\inf_{\mathcal{P}^{+}_{a,\mu}}E_{\mu}\leq m(a,\mu)$. Since $E_{\mu}>0$ on $\mathcal{P}^{-}_{a,\mu}$, we know that $\inf_{\mathcal{P}^{+}_{a,\mu}}E_{\mu}=\inf_{\mathcal{P}_{a,\mu}}E_{\mu}.$ Finally, by the continuity of $h$ there exists $\rho>0$ such that $h(t)\geq \frac{m(a,\mu)}{2}$ if $t\in [R_{0}-\rho,R_{0}]$. Therefore, $$E_{\mu}(u)\geq h(\|u\|_{D_{s}(\mathbb{R}^{N})})\geq \frac{m(a,\mu)}{2}>m(a,\mu)$$ for every $u\in S_{a}$ with $R_{0}-\rho\leq \|u\|_{D_{s}(\mathbb{R}^{N})}\leq R_{0}$.
This completes the proof.
\end{proof}
\section{$L^{2}$-critical perturbation}\label{sec3}
In this section, we consider the case $N>2s$  and $2<q=\overline{p}$.
Assume that
\begin{align}\label{U20}
\mu a^{\frac{4s}{N}}< \overline{p}\left(2C^{\overline{p}}_{N,\overline{p},s}\right)^{-1}.
\end{align}
We recall the decomposition of $$\mathcal{P}_{a,\mu}=\mathcal{P}^{+}_{a,\mu}\cup \mathcal{P}^{0}_{a,\mu}\cup \mathcal{P}^{-}_{a,\mu}.$$
\begin{lemma}\label{Lem9}
$\mathcal{P}^{0}_{a,\mu}=\emptyset$ and $\mathcal{P}_{a,\mu}$ is a smooth manifold of codimension 1 in $S_{a}$.
\end{lemma}
\begin{proof}
Assume by contradiction that if there exists a $u\in \mathcal{P}^{0}_{a,\mu}$, then
\begin{align}\label{U17}
||u||^{2}_{D_{s}(\mathbb{R}^{N})}-\mu\gamma_{q,s}\int_{\mathbb{R}^{N}} |u|^{q}dx-\int_{\mathbb{R}^{N}} |u|^{2_{s}^{\ast}}dx=0,
\end{align}
and
\begin{align}\label{U18}
2||u||^{2}_{D_{s}(\mathbb{R}^{N})} =\mu 2\gamma_{q,s}\int_{\mathbb{R}^{N}} |u|^{q}dx+2_{s}^{\ast}\int_{\mathbb{R}^{N}} |u|^{2_{s}^{\ast}}dx,
\end{align}
thus, from \eqref{U17} and \eqref{U18}, we have $\int_{\mathbb{R}^{N}} |u|^{2_{s}^{\ast}}dx=0$, which  is not possible since $u\in S_{a}.$ The rest of the proof is similar to that of Lemma \ref{Lem3}, so we omit the details here.
\end{proof}
\begin{lemma}\label{Lem8}
Under the condition of \eqref{U20}, for every $u\in S_{a}$, there is a unique $t_{u}\in \mathbb{R}$ such that $t_{u}\star u \in \mathcal{P}_{a,\mu},$  where $t_{u}$ is the unique critical point of the function of $\Psi^{\mu}_{u}$ and is a strict maximum point at positive level. Moreover,
\begin{enumerate}
\item[$(1)$] $\mathcal{P}_{a,\mu}=\mathcal{P}^{-}_{a,\mu}$.

\item[$(2)$] $\Psi^{\mu}_{u}(t)$ is strictly decreasing and concave on $(t_{u}, +\infty)$ and $t_{u}<0$ implies that $P_{\mu}(u)<0.$

\item[$(3)$] The map $u\in S_{a}: t_{u}\in \mathbb{R}$ os of class $C^{1}$.

\item[$(4)$] If $P_{\mu}(u)<0$, then $t_{u}<0$.
\end{enumerate}
\end{lemma}
\begin{proof}
Since
\begin{equation*}
 \Psi^{\mu}_{u}(t)=E_{\mu}(t\star u)=\left[\frac{1}{2}||u||^{2}_{D_{s}(\mathbb{R}^{N})}-\frac{\mu}{\overline{p}}\int_{\mathbb{R}^{N}} |u|^{\overline{p}}dx\right]e^{2st}-\frac{e^{2_{s}^{\ast}st}}{2_{s}^{\ast}}\int_{\mathbb{R}^{N}} |u|^{2_{s}^{\ast}}dx,
\end{equation*}
and  $t\star u\in \mathcal{P}_{a,\mu} $  if and only if $(\Psi^{\mu}_{u})'(t)=0$, it is easy to see that if  $\left[\frac{1}{2}||u||^{2}_{D_{s}(\mathbb{R}^{N})}-\frac{\mu}{\overline{p}}\int_{\mathbb{R}^{N}} |u|^{\overline{p}}dx\right]$ is positive, then $\Psi^{\mu}_{u}(t)$ has a unique critical point $t_{u}$, which is is a strict maximum point at positive level. By the fractional Gagliardo-Nirenberg-Sobolev inequality \eqref{U3}, we have $$\frac{1}{2}||u||^{2}_{D_{s}(\mathbb{R}^{N})}-\frac{\mu}{\overline{p}}\int_{\mathbb{R}^{N}} |u|^{\overline{p}}dx\geq \left(\frac{1}{2}-\frac{\mu}{\overline{p}}C^{\overline{p}}_{N,\overline{p},s}a^{\frac{4s}{N}}\right)||u||^{2}_{D_{s}(\mathbb{R}^{N})},$$ so under the condition of $\mu a^{\frac{4s}{N}}< \overline{p}(2C^{\overline{p}}_{N,\overline{p},s})^{-1},$ we know that $\frac{1}{2}||u||^{2}_{D_{s}(\mathbb{R}^{N})}-\frac{\mu}{\overline{p}}\int_{\mathbb{R}^{N}} |u|^{\overline{p}}dx>0.$ Since, if $u\in \mathcal{P}_{a,\mu}$, then $t_{u}$ is a maximum point, we have that $\Psi^{\mu}_{u}(0)\leq 0.$ Since $\mathcal{P}^{0}_{a,\mu}=\emptyset,$ we have $\Psi^{\mu}_{u}(0)<0.$ Thus, $\mathcal{P}_{a,\mu}=\mathcal{P}^{-}_{a,\mu}.$ To prove that the map $u\in S_{a}: t_{u}\in \mathbb{R}$ is of class $C^{1}$, we can apply the implicit function theorem as in Lemma \ref{Lem4}. Finally, since $(\Psi^{\mu}_{u})'(t)<0$ if and only if $t>t_{u}$, so $P_{\mu}(u)=(\Psi^{\mu}_{u})'(0)<0$ if and only if $t_{u}<0$.
\end{proof}
\begin{lemma}\label{Lem10}
$$m(a,\mu)=\inf_{\mathcal{P}_{a,\mu}}E_{\mu}>0.$$
\end{lemma}
\begin{proof}
If $u\in \mathcal{P}_{a,\mu},$ then $P_{\mu}(u)=0$, and then by fractional Gagliardo-Nirenberg-Sobolev inequality \eqref{U3} and Sobolev inequality \eqref{int2}, we have
\begin{align*}
||u||^{2}_{D_{s}(\mathbb{R}^{N})}=\mu\frac{2}{\overline{p}}\int_{\mathbb{R}^{N}} |u|^{\overline{p}}dx+\int_{\mathbb{R}^{N}} |u|^{2_{s}^{\ast}}dx
\leq \mu\frac{2}{\overline{p}}C^{\overline{p}}_{N,\overline{p},s}a^{\frac{4s}{N}}||u||^{2}_{D_{s}(\mathbb{R}^{N})}+S^{-\frac{2_{s}^{\ast}}{2}}_{s}||u||^{2_{s}^{\ast}}_{D_{s}(\mathbb{R}^{N})}
\end{align*}
From \eqref{U20} and above inequality, we have
\begin{align}\label{U19}
||u||^{2_{s}^{\ast}}_{D_{s}(\mathbb{R}^{N})}\geq S^{\frac{2_{s}^{\ast}}{2}}_{s}\left(1-\mu\frac{2}{\overline{p}}C^{\overline{p}}_{N,\overline{p},s}a^{\frac{4s}{N}}\right)||u||^{2}_{D_{s}(\mathbb{R}^{N})}\Rightarrow \inf_{\mathcal{P}_{a,\mu}}||u||_{D_{s}(\mathbb{R}^{N})}>0.
\end{align}
Thus, from $P_{\mu}(u)=0$ and above inequality, we have
\begin{align*}
E_{\mu}(u)
&=\frac{s}{N}\left(||u||^{2}_{D_{s}(\mathbb{R}^{N})}-\frac{2\mu}{\overline{p}}\int_{\mathbb{R}^{N}}|u|^{\overline{p}}dx\right)\\
&\geq  \frac{s}{N} \left(1-\mu\frac{2}{\overline{p}}C^{\overline{p}}_{N,\overline{p},s}a^{\frac{4s}{N}}\right)||u||^{2}_{D_{s}(\mathbb{R}^{N})}> 0.
\end{align*}
Therefore, $$m(a,\mu)=\inf_{\mathcal{P}_{a,\mu}}E_{\mu}>0.$$
As required.
\end{proof}
\begin{lemma}\label{Lem11}
There exists $k>0$ sufficiently small such that $$0<\sup_{\overline{A_{k}}}E_{\mu}<m(a,\mu)\ \text{and}\ u\in \overline{A_{k}}\Rightarrow E_{\mu}(u), P_{\mu}(u)>0.$$
where $ A_{k}=\left\{u\in S_{a}: ||u||^{2}_{D_{s}(\mathbb{R}^{N})}<k\right\}.$
\end{lemma}
\begin{proof}
By fractional Gagliardo-Nirenberg-Sobolev inequality \eqref{U3} and Sobolev inequality \eqref{int2}, we have
\begin{align*}
E_{\mu}(u)
\geq \left(\frac{1}{2}-\frac{\mu}{\overline{p}}C^{\overline{p}}_{N,\overline{p},s}a^{\frac{4s}{N}}\right)||u||^{2}_{D_{s}(\mathbb{R}^{N})}-\frac{1}{2_{s}^{\ast}}S^{-\frac{2_{s}^{\ast}}{2}}_{s}||u||^{2_{s}^{\ast}}_{D_{s}(\mathbb{R}^{N})}>0,
\end{align*}
and
\begin{align*}
P_{\mu}(u)&=s||u||^{2}_{D_{s}(\mathbb{R}^{N})}-s\mu\frac{2}{\overline{p}}\int_{\mathbb{R}^{N}} |u|^{\overline{p}}dx-s\int_{\mathbb{R}^{N}} |u|^{2_{s}^{\ast}}dx\\
& \geq s\left(1-\frac{2\mu}{\overline{p}}C^{\overline{p}}_{N,\overline{p},s}a^{\frac{4s}{N}}\right)||u||^{2}_{D_{s}(\mathbb{R}^{N})}-\frac{s}{2_{s}^{\ast}}S^{-\frac{2_{s}^{\ast}}{2}}_{s}||u||^{2_{s}^{\ast}}_{D_{s}(\mathbb{R}^{N})}>0,
\end{align*}
provided that $u\in \overline{A_{k}}$ for $k$ small enough.  By Lemma \ref{Lem11}, we know that $m(a,\mu)>0$, thus if necessary replacing $k$ with smaller quantity, we also have $$E_{\mu}(u)\leq \frac{1}{2}||u||^{2}_{D_{s}(\mathbb{R}^{N})}< m(a,\mu).$$
The proof is complete.
\end{proof}
In order to apply Proposition \ref{LM5} and recover compactness, we need an estimate from above on $m_{r}(a,\mu)=\inf_{\mathcal{P}_{a,\mu}\bigcap S^{r}_{a}}E_{\mu}$, where $S^{r}_{a}$ is the subset of the radial functions in $S_{a}.$

\begin{lemma}\label{Lem16}
Under  condition \eqref{U20}, we have $m_{r}(a,\mu)< \frac{s}{N}S^{\frac{N}{2s}}_{s}.$
\end{lemma}
\begin{proof}
From \cite{CT}, we know that $S_{s}$\ is attained in\ $\mathbb{R}^{N}$ by
$$
U_0(x)=C(N,s)\left(\frac{1}{1+|x|^{2}}\right)^{\frac{N-2s}{2}}
$$ with $C(N,s)$ chosen so that

$$
\|U_0(x)\|^{2}_{D_{s}(\mathbb{R}^{N})}=\int_{\mathbb{R}^{N}}|U_0(x)|^{2_{s}^{\ast}}dx=S^{\frac{N}{2s}}_{s}.
$$

Take $\eta(x)\in C^{\infty}_{0}(\mathbb{R}^{N},[0,1])$ be a cut-off function such that $0\leq\eta\leq1,\eta=1\ \text{on}\ B(0,\delta)$ and $\eta=1 \text{ on }\ \mathbb{R}^{N}\setminus B(0,2\delta)$. Let

$$
u_{\epsilon}=\eta(x)U_{\epsilon}(x),\ \  \ v_{\epsilon}=a\frac{u_{\epsilon}}{\|u_{\epsilon}\|_{L^{2}(\mathbb{R}^{N})}}$$

and $$U_{\epsilon}(x)=\epsilon^{-\frac{N-2s}{2}}U_{0}(\frac{x}{\epsilon}).
$$
From Proposition 21 and Proposition 22 in \cite{SV1}, it is easy to deduce that the following estimates hold true
\begin{equation}\label{int6}
\|v_{\epsilon}\|^{2}_{D_{s}(\mathbb{R}^{N})}\leq \|U_{0}\|^{2}_{D_{s}(\mathbb{R}^{N})}+o(\epsilon^{N-2s}),
\end{equation}
\begin{equation}\label{int7}
\int_{\mathbb{R}^{N}}|v_{\epsilon}|^{2}dx=\left\{
\begin{array}{rcl}
C\epsilon^{2s}+o(\epsilon^{N-2s}),       &      & if\  N>4s,\\
C\epsilon^{2s}\log(\frac{1}{\epsilon})+o(\epsilon^{2s}),    &      & if\  N=4s,\\
 C\epsilon^{N-2s}+o(\epsilon^{2s}),    &      & if\  N<4s,
\end{array} \right.
\end{equation}
It is easy to see that
\begin{equation}\label{int9}
\int_{\mathbb{R}^{N}}|v_{\epsilon}|^{^{2_{s}^{\ast}}}dx= \int_{\mathbb{R}^{N}}|U_{0}|^{^{2_{s}^{\ast}}}dx+o(\epsilon^{N}).
\end{equation}
By the similar arguments as Proposition 22 in \cite{SV1}, we can deduce that
\begin{equation}\label{int12}
\int_{\mathbb{R}^{N}}|v_{\epsilon}|^{\overline{p}}dx=\left\{
\begin{array}{rcl}
&C\epsilon^{N-\frac{ N-2 s}{2}\overline{p}}+o(\epsilon^{\frac{N-2s}{2}\overline{p}}),             & if\ N>\frac{\overline{p}}{\overline{p}-1}2s,\\
&C\epsilon^{\frac{N}{2}}\log(\frac{1}{\epsilon})+o(\epsilon^{\frac{N}{2}}),          & if\  N=\frac{\overline{p}}{\overline{p}-1}2s,\\
&C\epsilon^{\frac{N-2s}{2}\overline{p}}+o(\epsilon^{N-\frac{ N-2 s}{2}\overline{p}}),          & if\  N<\frac{\overline{p}}{\overline{p}-1}2s.
\end{array} \right.
\end{equation}
It is easy to see that $u_{\epsilon}\in C^{\infty}_{0}(\mathbb{R}^{N},[0,1])$ and $v_{\epsilon}\in S_{a}^{r}.$ By Lemma \ref{Lem8}, we know that $$m_{r}(a,\mu)=\inf_{\mathcal{P}_{a,\mu}\bigcap S^{r}_{a}}E_{\mu}\leq E_{\mu}(t_{v_{\epsilon}}\star v_{\epsilon})=\max_{t\in \mathbb{R}}E_{\mu}(t\star v_{\epsilon}).$$
Next, we give a upper estimate of  $$E_{\mu}(t_{v_{\epsilon}}\star v_{\epsilon})=\max_{t\in \mathbb{R}}E_{\mu}(t\star v_{\epsilon}).$$

{\bf Step 1.} Consider the case $\mu=0$  and estimate $$\max_{t\in \mathbb{R}}\Psi^{0}_{v_{\epsilon}}(t)=E_{0}(t\star v_{\epsilon}).$$
Since \begin{equation*}
 \Psi^{0}_{v_{\epsilon}}(t)=\frac{e^{2st}}{2}||v_{\epsilon}||^{2}_{D_{s}(\mathbb{R}^{N})}-\frac{e^{2_{s}^{\ast}st}}{2_{s}^{\ast}}\int_{\mathbb{R}^{N}} |v_{\epsilon}|^{2_{s}^{\ast}}dx.
\end{equation*}
It is easy to see that for every $v_{\epsilon}\in S_{a}$ the function $\Psi^{0}_{v_{\epsilon}}(t)$ has a unique critical point $t_{v_{\epsilon},0}$, which is  a strict maximum point and is given by
\begin{align}\label{U21}
e^{st_{v_{\epsilon},0}}=\left(\frac{||v_{\epsilon}||^{2}_{D_{s}(\mathbb{R}^{N})}}{\int_{\mathbb{R}^{N}} |v_{\epsilon}|^{2_{s}^{\ast}}dx}\right)^{\frac{1}{2_{s}^{\ast}-2}}.
\end{align}
Thus, from  the estimates \eqref{int6}--\eqref{int9}, we have
\begin{align*}
 \max_{t\in \mathbb{R}}E_{0}(t\star v_{\epsilon})&=E_{0}(t_{v_{\epsilon},0}\star v_{\epsilon})=\Psi^{0}_{v_{\epsilon}}(t_{v_{\epsilon},0})\\
 &= \left[\frac{1}{2}\left(\frac{||v_{\epsilon}||^{2}_{D_{s}(\mathbb{R}^{N})}}{\int_{\mathbb{R}^{N}} |v_{\epsilon}|^{2_{s}^{\ast}}dx}\right)^{\frac{2}{2_{s}^{\ast}-2}}||v_{\epsilon}||^{2}_{D_{s}(\mathbb{R}^{N})}-\frac{1}{2_{s}^{\ast}}\left(\frac{||v_{\epsilon}||^{2}_{D_{s}(\mathbb{R}^{N})}}{\int_{\mathbb{R}^{N}} |v_{\epsilon}|^{2_{s}^{\ast}}dx}\right)^{\frac{2_{s}^{\ast}}{2_{s}^{\ast}-2}}\int_{\mathbb{R}^{N}} |v_{\epsilon}|^{2_{s}^{\ast}}dx\right]\\
 &= \frac{s}{N}\left(\frac{||v_{\epsilon}||^{2}_{D_{s}(\mathbb{R}^{N})}}{\left(\int_{\mathbb{R}^{N}} |v_{\epsilon}|^{2_{s}^{\ast}}dx\right)^{\frac{2}{2_{s}^{\ast}}}}\right)^{\frac{2_{s}^{\ast}}{2_{s}^{\ast}-2}}
 =\frac{s}{N}\left[\frac{S^{\frac{N}{2s}}_{s}+O(\epsilon^{N-2s})}{\left(S^{\frac{N}{2s}}_{s}+O(\epsilon^{N})\right)^{\frac{2}{2_{s}^{\ast}}}}\right]^{\frac{N}{2s}}=\frac{s}{N}S^{\frac{N}{2s}}_{s}+O(\epsilon^{N-2s})
 \end{align*}
 {\bf Step 2.} Estimate on $t_{v_{\epsilon}, \mu}$.  Since $$\Psi^{\mu}_{v_{\epsilon}}(t)=E_{\mu}(t\star v_{\epsilon})=\frac{e^{2st}}{2}||v_{\epsilon}||^{2}_{D_{s}(\mathbb{R}^{N})}-\mu\frac{e^{2st}}{\overline{p}}\int_{\mathbb{R}^{N}} |v_{\epsilon}|^{\overline{p}}dx-\frac{e^{2_{s}^{\ast}st}}{2_{s}^{\ast}}\int_{\mathbb{R}^{N}} |v_{\epsilon}|^{2_{s}^{\ast}}dx.$$ Let $t_{v_{\epsilon},\mu}$ be the unique maximum point of $\Psi^{\mu}_{v_{\epsilon}}(t)$, then by $(\Psi^{\mu}_{v_{\epsilon}})'(t)=P_{\mu}(t_{v_{\epsilon},\mu}\star v_{\epsilon})=0$ and fractional Gagliardo-Nirenberg-Sobolev inequality \eqref{U3}, we have
 $$e^{(2_{s}^{\ast}-2)st}=\frac{||v_{\epsilon}||^{2}_{D_{s}(\mathbb{R}^{N})}}{\int_{\mathbb{R}^{N}} |v_{\epsilon}|^{2_{s}^{\ast}}dx}-\frac{2\mu}{\overline{p}}\frac{\int_{\mathbb{R}^{N}} |v_{\epsilon}|^{\overline{p}}dx}{\int_{\mathbb{R}^{N}} |v_{\epsilon}|^{2_{s}^{\ast}}dx}\geq \left(1-\frac{2\mu}{\overline{p}}C^{\overline{p}}_{N,\overline{p},s}a^{\frac{4s}{N}}\right)\frac{||v_{\epsilon}||^{2}_{D_{s}(\mathbb{R}^{N})}}{\int_{\mathbb{R}^{N}} |v_{\epsilon}|^{2_{s}^{\ast}}dx}.$$
 {\bf Step 3.} Estimate on $\max_{t\in \mathbb{R}}\Psi^{\mu}_{v_{\epsilon}}(t)$. Since
 \begin{align*}
 &\max_{t\in \mathbb{R}}\Psi^{\mu}_{v_{\epsilon}}(t)=\Psi^{\mu}_{v_{\epsilon}}(t_{v_{\epsilon},\mu})=\Psi^{0}_{v_{\epsilon}}(t_{v_{\epsilon},\mu})-\mu\frac{e^{2st_{v_{\epsilon},\mu}}}{\overline{p}}\int_{\mathbb{R}^{N}} |v_{\epsilon}|^{\overline{p}}dx\\
 &\leq \sup_{\mathbb{R}}\Psi^{0}_{v_{\epsilon}}-\frac{\mu}{\overline{p}}\left(1-\frac{2\mu}{\overline{p}}C^{\overline{p}}_{N,\overline{p},s}a^{\frac{4s}{N}}\right)^{\frac{2}{2_{s}^{\ast}-2}}\left(\frac{||v_{\epsilon}||^{2}_{D_{s}(\mathbb{R}^{N})}}{\int_{\mathbb{R}^{N}} |v_{\epsilon}|^{2_{s}^{\ast}}dx}\right)^{\frac{2}{2_{s}^{\ast}-2}}\int_{\mathbb{R}^{N}} |v_{\epsilon}|^{\overline{p}}dx\\
 &\leq \frac{s}{N}S^{\frac{N}{2s}}_{s}+O(\epsilon^{N-2s})-\frac{\mu}{\overline{p}}\left(1-\frac{2\mu}{\overline{p}}C^{\overline{p}}_{N,\overline{p},s}a^{\frac{4s}{N}}\right)^{\frac{2}{2_{s}^{\ast}-2}}\frac{a^{\frac{4s}{N}}}{\|u_{\epsilon}\|^{\frac{4s}{N}}_{L^{2}(\mathbb{R}^{N})}}\frac{||u_{\epsilon}||^{\frac{4}{2_{s}^{\ast}-2}}_{D_{s}(\mathbb{R}^{N})}\int_{\mathbb{R}^{N}} |u_{\epsilon}|^{\overline{p}}dx}{\left(\int_{\mathbb{R}^{N}} |u_{\epsilon}|^{2_{s}^{\ast}}dx\right)^{\frac{2}{2_{s}^{\ast}-2}}}\\
 &\leq \frac{s}{N}S^{\frac{N}{2s}}_{s}+O(\epsilon^{N-2s})-C_{N,a,\mu} \frac{\int_{\mathbb{R}^{N}} |u_{\epsilon}|^{\overline{p}}dx}{\|u_{\epsilon}\|^{\frac{4s}{N}}_{L^{2}(\mathbb{R}^{N})}}
 \end{align*}
From \eqref{int7} and \eqref{int12}, we have the following estimate:
 \begin{align*}
&\frac{\int_{\mathbb{R}^{N}} |u_{\epsilon}|^{\overline{p}}dx}{\|u_{\epsilon}\|^{\frac{4s}{N}}_{L^{2}(\mathbb{R}^{N})}}
 = \left\{
\begin{aligned}
&C\epsilon^{N-\frac{N-2s}{2}\overline{p}-\frac{4s^{2}}{N}}=C,  &\text{if}\ &N>4s, \\
&C\epsilon^{4s-s\overline{p}-s}|\ln\epsilon|^{-\frac{1}{2}}=C|\ln \epsilon|^{-\frac{1}{2}}, &\text{if}\ &N=4s, \\
&C\epsilon^{N-\frac{N-2s}{2}\overline{p}-\frac{N-2s}{2}\frac{4s}{N}}=C\epsilon^{\frac{2s(4s-N)}{N}}, &\text{if} \  &\frac{\overline{p}}{\overline{p}-1}2s<N<4s,\\
&C\epsilon^{\frac{N}{2}-\frac{N-2s}{2}\frac{4s}{N}}|\ln \epsilon| , &\text{if} \   &N=\frac{\overline{p}}{\overline{p}-1}2s,\\
&C\epsilon^{\frac{N-2s}{2}\overline{p}-\frac{N-2s}{2}\frac{4s}{N}}=C\epsilon^{N-2s}, &\text{if} \   &2s<N<\frac{\overline{p}}{\overline{p}-1}2s.
\end{aligned}
\right.
\end{align*}
Thus, $$\max_{t\in \mathbb{R}}\Psi^{\mu}_{v_{\epsilon}}(t)\leq \frac{s}{N}S^{\frac{N}{2s}}_{s}.$$
The proof is thus finished.
\end{proof}
\section{$L^{2}$-supercritical perturbation}\label{sec4}
In this section, we consider $N>2s$  and $\overline{p}<q<2_{s}^{\ast}$. We recall the decomposition of $$\mathcal{P}_{a,\mu}=\mathcal{P}^{+}_{a,\mu}\cup \mathcal{P}^{0}_{a,\mu}\cup \mathcal{P}^{-}_{a,\mu}.$$
\begin{lemma}\label{Lem12}
$\mathcal{P}^{0}_{a,\mu}=\emptyset$ and $\mathcal{P}_{a,\mu}$ is a smooth manifold of codimension 1 in $S_{a}$.
\end{lemma}
\begin{proof}
Assume by contradiction that  there exists a $u\in \mathcal{P}^{0}_{a,\mu}$, then
\begin{align}\label{U22}
||u||^{2}_{D_{s}(\mathbb{R}^{N})}-\mu\gamma_{q,s}\int_{\mathbb{R}^{N}} |u|^{q}dx-\int_{\mathbb{R}^{N}} |u|^{2_{s}^{\ast}}dx=0,
\end{align}
and
\begin{align}\label{U23}
2||u||^{2}_{D_{s}(\mathbb{R}^{N})} =\mu q\gamma^{2}_{q,s}\int_{\mathbb{R}^{N}} |u|^{q}dx+2_{s}^{\ast}\int_{\mathbb{R}^{N}} |u|^{2_{s}^{\ast}}dx.
\end{align}
Thus, from \eqref{U22} and \eqref{U23}, we have $$(2-q\gamma_{q,s})\mu \gamma_{q,s}\int_{\mathbb{R}^{N}} |u|^{q}dx=(2_{s}^{\ast}-2)\int_{\mathbb{R}^{N}} |u|^{2_{s}^{\ast}}dx=0.$$
Since $2-q\gamma_{q,s}<0, 2_{s}^{\ast}-2>0$, we have $u=0,$ which  is not possible, thanks to $u\in S_{a}.$ The rest of the proof is similar to the one of Lemma \ref{Lem3}, so we omit the details here.
\end{proof}
\begin{lemma}\label{Lem13}
For every $u\in S_{a}$, there is a unique $t_{u}\in \mathbb{R}$ such that $t_{u}\star u \in \mathcal{P}_{a,\mu},$  where $t_{u}$ is the unique critical point of the function of $\Psi^{\mu}_{u}$ and is a strict maximum point at positive level, moreover,
\begin{enumerate}
 \item[$(1)$] $\mathcal{P}_{a,\mu}=\mathcal{P}^{-}_{a,\mu}$.

 \item[$(2)$] $\Psi^{\mu}_{u}(t)$ is strictly decreasing and concave on $(t_{u}, +\infty)$ and $t_{u}<0$ implies that $P_{\mu}(u)<0.$

 \item[$(3)$] The map $u\in S_{a}: t_{u}\in \mathbb{R}$ os of class $C^{1}$.

 \item[$(4)$] If $P_{\mu}(u)<0$, then $t_{u}<0$.
\end{enumerate}
\end{lemma}
\begin{proof}
Since
\begin{equation*}
 \Psi^{\mu}_{u}(t)=E_{\mu}(t\star u)=\frac{e^{2st}}{2}||u||^{2}_{D_{s}(\mathbb{R}^{N})}-\mu\frac{e^{q\gamma_{q,s}st}}{q}\int_{\mathbb{R}^{N}} |u|^{q}dx-\frac{e^{2_{s}^{\ast}st}}{2_{s}^{\ast}}\int_{\mathbb{R}^{N}} |u|^{2_{s}^{\ast}}dx,
\end{equation*}
  and $$(\Psi^{\mu}_{u})'(t)=se^{2st}||u||^{2}_{D_{s}(\mathbb{R}^{N})}-\mu\gamma_{q,s}se^{q\gamma_{q,s}st}\int_{\mathbb{R}^{N}} |u|^{q}dx-se^{2_{s}^{\ast}st} \int_{\mathbb{R}^{N}} |u|^{2_{s}^{\ast}}dx,$$ it follows that $(\Psi^{\mu}_{u})'(t)=0$ if and only if $$||u||^{2}_{D_{s}(\mathbb{R}^{N})}= f(t):=\mu\gamma_{q,s}e^{(q\gamma_{q,s}-2)st}\int_{\mathbb{R}^{N}} |u|^{q}dx+e^{(2_{s}^{\ast}-2)st} \int_{\mathbb{R}^{N}} |u|^{2_{s}^{\ast}}dx.$$
It is easy to see that $f(t)$ is positive , continuous, monotone increasing and $f(t)\rightarrow 0^{+}$ as $t\rightarrow -\infty$ and $f(t)\rightarrow +\infty$ as $t\rightarrow +\infty$. Thus, there exists a unique point $t_{u,s}$ such that $f(t)=||u||^{2}_{D_{s}(\mathbb{R}^{N})}$. Since $\Psi^{\mu}_{u}\rightarrow 0^{+}$ as $s\rightarrow -\infty$ and  $\Psi^{\mu}_{u}\rightarrow -\infty$ as $s\rightarrow +\infty$, we know that there is a unique $t_{u}\in \mathbb{R}$ such that $t_{u}\star u \in \mathcal{P}_{a,\mu},$  where $t_{u}$ is the unique critical point of the function of $\Psi^{\mu}_{u}$ and is a strict maximum point at positive level. Since $t_{u}$ is a strict maximum point, we know that $(\Psi^{\mu}_{u})''(t_{u})\leq 0.$  Because $\mathcal{P}^{0}_{a,\mu}=\emptyset,$ we have $(\Psi^{\mu}_{u})''(t_{u})\neq 0,$ which implies that $t_{u}\star u \in \mathcal{P}^{-}_{a,\mu}$, since $\Psi^{\mu}_{u}(t)$ has exactly one maximum point, so $\mathcal{P}_{a,\mu}=\mathcal{P}^{-}_{a,\mu}.$ To prove that the map $u\in S_{a}: t_{u}\in \mathbb{R}$ os of class $C^{1}$, we can apply the implicit function theorem as Lemma \ref{Lem4}. Finally, since $(\Psi^{\mu}_{u})'(t)<0$ if and only if $t>t_{u}$, so $P_{\mu}(u)=(\Psi^{\mu}_{u})'(0)<0$ if and only if $t_{u}<0$.
\end{proof}
\begin{lemma}\label{Lem14} There holds
$$m(a,\mu)=\inf_{\mathcal{P}_{a,\mu}}E_{\mu}>0.$$
\end{lemma}
\begin{proof}
If $u\in \mathcal{P}_{a,\mu},$ then $P_{\mu}(u)=0$, then by fractional Gagliardo-Nirenberg-Sobolev inequality \eqref{U3} and Sobolev inequality \eqref{int2}, we have
\begin{align*}
||u||^{2}_{D_{s}(\mathbb{R}^{N})}&=\mu\gamma_{q,s}\int_{\mathbb{R}^{N}} |u|^{q}dx+\int_{\mathbb{R}^{N}} |u|^{2_{s}^{\ast}}dx\\
&\leq \mu\gamma_{q,s}C^{q}_{N,q,s}a^{(1-\gamma_{q,s})q}||u||^{q\gamma_{q,s}}_{D_{s}(\mathbb{R}^{N})}+S^{-\frac{2_{s}^{\ast}}{2}}_{s}||u||^{2_{s}^{\ast}}_{D_{s}(\mathbb{R}^{N})}
\end{align*}
Thus, from above inequality and $||u||^{2}_{D_{s}(\mathbb{R}^{N})}\neq 0$ (since $u\in S_{a}$), we have
\begin{align*}
\mu\gamma_{q,s}C^{q}_{N,q,s}a^{(1-\gamma_{q,s})q}||u||^{q\gamma_{q,s}-2}_{D_{s}(\mathbb{R}^{N})}+S^{-\frac{2_{s}^{\ast}}{2}}_{s}||u||^{2_{s}^{\ast}-2}_{D_{s}(\mathbb{R}^{N})}\geq 1, \ \forall \ u\in \mathcal{P}_{a,\mu},
\end{align*}
which  implies that $\inf_{u\in \mathcal{P}_{a,\mu}}||u||_{D_{s}(\mathbb{R}^{N})}>0$. Since $$\mu\gamma_{q,s}\int_{\mathbb{R}^{N}} |u|^{q}dx+\int_{\mathbb{R}^{N}} |u|^{2_{s}^{\ast}}dx=||u||^{2}_{D_{s}(\mathbb{R}^{N})},$$ we have $$\inf_{u\in \mathcal{P}_{a,\mu}}\left[\mu\gamma_{q,s}\int_{\mathbb{R}^{N}} |u|^{q}dx+\int_{\mathbb{R}^{N}} |u|^{2_{s}^{\ast}}dx\right]>0.$$
Thus, from $P_{\mu}(u)=0$ and above inequality, we have
\begin{align*}
\inf_{u\in \mathcal{P}_{a,\mu}}E_{\mu}(u)&=\inf_{u\in \mathcal{P}_{a,\mu}}\left[\frac{1}{2}||u||^{2}_{D_{s}(\mathbb{R}^{N})}
-\frac{\mu}{q}\int_{\mathbb{R}^{N}}|u|^{q}dx-\frac{1}{2_{s}^{\ast}}\int_{\mathbb{R}^{N}}|u|^{2_{s}^{\ast}}dx\right]\\
&=\inf_{u\in \mathcal{P}_{a,\mu}}\left[\frac{\mu}{q}\left(\frac{q\gamma_{q,s}}{2}-1\right)\int_{\mathbb{R}^{N}}|u|^{q}dx+\frac{s}{N}\int_{\mathbb{R}^{N}}|u|^{2_{s}^{\ast}}dx\right]> 0.
\end{align*}
Therefore, $$m(a,\mu)=\inf_{\mathcal{P}_{a,\mu}}E_{\mu}>0.$$
This finishes the proof.
\end{proof}
\begin{lemma}\label{Lem15}
There exists $k>0$ sufficiently small such that $$0<\sup_{\overline{A_{k}}}E_{\mu}<m(a,\mu)\ \text{and}\ u\in \overline{A_{k}}\Rightarrow E_{\mu}(u), P_{\mu}(u)>0,$$
where $ A_{k}=\left\{u\in S_{a}: ||u||^{2}_{D_{s}(\mathbb{R}^{N})}<k\right\}.$
\end{lemma}
\begin{proof}
By fractional Gagliardo-Nirenberg-Sobolev inequality \eqref{U3} and Sobolev inequality \eqref{int2}, we have
\begin{align*}
E_{\mu}(u)\geq \frac{1}{2}||u||^{2}_{D_{s}(\mathbb{R}^{N})}-\frac{\mu}{q}C^{q}_{N,q,s}a^{q(1-\gamma_{q,s})}||u||^{q\gamma_{q,s}}_{D_{s}(\mathbb{R}^{N})}-\frac{1}{2_{s}^{\ast}}S^{-\frac{2_{s}^{\ast}}{2}}_{s}||u||^{2_{s}^{\ast}}_{D_{s}(\mathbb{R}^{N})}>0,
\end{align*}
and
\begin{align*}
P_{\mu}(u)&=s||u||^{2}_{D_{s}(\mathbb{R}^{N})}-s\mu\gamma_{q,s}\int_{\mathbb{R}^{N}} |u|^{q}dx-s\int_{\mathbb{R}^{N}} |u|^{2_{s}^{\ast}}dx\\
& \geq s||u||^{2}_{D_{s}(\mathbb{R}^{N})}-s\mu\gamma_{q,s}C^{q}_{N,q,s}a^{q(1-\gamma_{q,s})}||u||^{q\gamma_{q,s}}_{D_{s}(\mathbb{R}^{N})}-sS^{-\frac{2_{s}^{\ast}}{2}}_{s}||u||^{2_{s}^{\ast}}_{D_{s}(\mathbb{R}^{N})}>0.
\end{align*}
If $u\in \overline{A_{k}}$ for $k$ small enough.  By Lemma \ref{Lem14}, we know that $m(a,\mu)>0$, thus if necessary replacing $k$ with smaller quantity, we also have $$E_{\mu}(u)\leq \frac{1}{2}||u||^{2}_{D_{s}(\mathbb{R}^{N})}< m(a,\mu).$$
This ends the proof.
\end{proof}
In order to apply Proposition \ref{LM5} and recover compactness, we need an estimate from above on $m_{r}(a,\mu)=\inf_{\mathcal{P}_{a,\mu}\bigcap S^{r}_{a}}E_{\mu}$, where $S^{r}_{a}$ is the subset of the radial functions in $S_{a}.$
\begin{lemma}\label{Lem17}
If one of following conditions holds:
\begin{enumerate}
 \item[$(1)$] $ N>4s\  \text{and}\  \mu a^{q(1-\gamma_{q,s})}<\frac{S^{\frac{N}{4s}q(1-\gamma_{q,s})}_{s}}{\gamma_{q,s}}$;

 \item[$(2)$] $ N=\frac{q}{q-1}2s\ \text{and}\ \mu a^{q(1-\gamma_{q,s})}<\frac{S^{\frac{N}{4s}q(1-\gamma_{q,s})}_{s}}{\gamma_{q,s}}$;

 \item[$(3)$] $N=4s \  \text{or}\ \frac{q}{q-1}2s<N<4s\  \text{or}\ 2s<N<\frac{q}{q-1}2s $,
 \end{enumerate}
 then we have $m_{r}(a,\mu)< \frac{s}{N}S^{\frac{N}{2s}}_{s}.$
\end{lemma}
\begin{proof}
Let us recall the definition of $u_{\epsilon}$ and $v_{\epsilon}$ as Lemma \ref{Lem16}. It is easy to see that $u_{\epsilon}\in C^{\infty}_{0}(\mathbb{R}^{N},[0,1])$ and $v_{\epsilon}\in S_{a}^{r}.$ By Lemma \ref{Lem8}, we know that $$m_{r}(a,\mu)=\inf_{\mathcal{P}_{a,\mu}\bigcap S^{r}_{a}}E_{\mu}\leq E_{\mu}(t_{v_{\epsilon},\mu}\star v_{\epsilon})=\max_{t\in \mathbb{R}}E_{\mu}(t\star v_{\epsilon}).$$

By the same argument as step 1 in Lemma \ref{Lem16}, we have $$\Psi^{0}_{v_{\epsilon}}(t_{v_{\epsilon},0})=\frac{s}{N}S^{\frac{N}{2s}}_{s}+O(\epsilon^{N-2s}).$$
{\bf Step 1.} Estimate on $t_{v_{\epsilon}, \mu}$. Since  $$\Psi^{\mu}_{v_{\epsilon}}(t)=E_{\mu}(t\star v_{\epsilon})=\frac{e^{2st}}{2}||v_{\epsilon}||^{2}_{D_{s}(\mathbb{R}^{N})}-\mu\frac{e^{q\gamma_{q,s}st}}{q}\int_{\mathbb{R}^{N}} |v_{\epsilon}|^{q}dx-\frac{e^{2_{s}^{\ast}st}}{2_{s}^{\ast}}\int_{\mathbb{R}^{N}} |v_{\epsilon}|^{2_{s}^{\ast}}dx $$ and $t_{v_{\epsilon},\mu}$ be the unique maximum point of $\Psi^{\mu}_{v_{\epsilon}}(t)$, then by $(\Psi^{\mu}_{v_{\epsilon}})'(t)=P_{\mu}(t_{v_{\epsilon},\mu}\star v_{\epsilon})=0,$ we have$$e^{2_{s}^{\ast}st_{v_{\epsilon}, \mu}} \int_{\mathbb{R}^{N}} |v_{\epsilon}|^{2_{s}^{\ast}}dx=e^{2st_{v_{\epsilon}, \mu}}||v_{\epsilon}||^{2}_{D_{s}(\mathbb{R}^{N})}-\mu\gamma_{q,s}e^{q\gamma_{q,s}st_{v_{\epsilon}, \mu}}\int_{\mathbb{R}^{N}} |v_{\epsilon}|^{q}dx\leq e^{2st_{v_{\epsilon}, \mu}}||v_{\epsilon}||^{2}_{D_{s}(\mathbb{R}^{N})},$$ which means that
\begin{align}\label{U24}
e^{st_{v_{\epsilon}, \mu}}\leq \left(\frac{||v_{\epsilon}||^{2}_{D_{s}(\mathbb{R}^{N})}}{\int_{\mathbb{R}^{N}} |v_{\epsilon}|^{2_{s}^{\ast}}dx}\right)^{\frac{1}{2_{s}^{\ast}-2}}.
\end{align}
By \eqref{U24}, $q\gamma_{q,s}>2$ and $v_{\epsilon}=\frac{au_{\epsilon}}{\|u_{\epsilon}\|_{L^{2}(\mathbb{R}^{N})}}$, we have
\begin{align}\label{U25}
&e^{(2_{s}^{\ast}-2)st_{v_{\epsilon}, \mu}}\\\nonumber
&=\frac{||v_{\epsilon}||^{2}_{D_{s}(\mathbb{R}^{N})}}{\int_{\mathbb{R}^{N}} |v_{\epsilon}|^{2_{s}^{\ast}}dx}-\mu\gamma_{q,s}e^{(q\gamma_{q,s}-2)st_{v_{\epsilon}, \mu}}\frac{\int_{\mathbb{R}^{N}} |v_{\epsilon}|^{q}dx}{\int_{\mathbb{R}^{N}} |v_{\epsilon}|^{2_{s}^{\ast}}dx}\\\nonumber
&\geq \frac{||v_{\epsilon}||^{2}_{D_{s}(\mathbb{R}^{N})}}{\int_{\mathbb{R}^{N}} |v_{\epsilon}|^{2_{s}^{\ast}}dx}-\mu\gamma_{q,s}\frac{\int_{\mathbb{R}^{N}} |v_{\epsilon}|^{q}dx}{\int_{\mathbb{R}^{N}} |v_{\epsilon}|^{2_{s}^{\ast}}dx}\left(\frac{||v_{\epsilon}||^{2}_{D_{s}(\mathbb{R}^{N})}}{\int_{\mathbb{R}^{N}} |v_{\epsilon}|^{2_{s}^{\ast}}dx}\right)^{\frac{q\gamma_{q,s}-2}{2_{s}^{\ast}-2}}\\\nonumber
&\geq\frac{\|u_{\epsilon}\|^{2_{s}^{\ast}-2}_{L^{2}(\mathbb{R}^{N})}}{a^{2_{s}^{\ast}-2}} \frac{||u_{\epsilon}||^{2}_{D_{s}(\mathbb{R}^{N})}}{\int_{\mathbb{R}^{N}} |u_{\epsilon}|^{2_{s}^{\ast}}dx}-\mu\gamma_{q,s}\frac{\|u_{\epsilon}\|^{2_{s}^{\ast}-q}_{L^{2}(\mathbb{R}^{N})}}{a^{2_{s}^{\ast}-q}}\frac{\int_{\mathbb{R}^{N}} |u_{\epsilon}|^{q}dx}{\int_{\mathbb{R}^{N}} |u_{\epsilon}|^{2_{s}^{\ast}}dx}\left(\frac{\|u_{\epsilon}\|^{2_{s}^{\ast}-2}_{L^{2}(\mathbb{R}^{N})}}{a^{2_{s}^{\ast}-2}} \frac{||u_{\epsilon}||^{2}_{D_{s}(\mathbb{R}^{N})}}{\int_{\mathbb{R}^{N}} |u_{\epsilon}|^{2_{s}^{\ast}}dx}\right)^{\frac{q\gamma_{q,s}-2}{2_{s}^{\ast}-2}}\\\nonumber
&\geq\frac{\|u_{\epsilon}\|^{2_{s}^{\ast}-2}_{L^{2}(\mathbb{R}^{N})}}{a^{2_{s}^{\ast}-2}} \frac{\left(||u_{\epsilon}||^{2}_{D_{s}(\mathbb{R}^{N})}\right)^{\frac{q\gamma_{q,s}-2}{2_{s}^{\ast}-2}}}{\int_{\mathbb{R}^{N}} |u_{\epsilon}|^{2_{s}^{\ast}}dx}\left[\left(||u_{\epsilon}||^{2}_{D_{s}(\mathbb{R}^{N})}\right)^{\frac{2_{s}^{\ast}-q\gamma_{q,s}}{2_{s}^{\ast}-2}}-\frac{\mu\gamma_{q,s}a^{q(1-\gamma_{q,s})}\int_{\mathbb{R}^{N}} |u_{\epsilon}|^{q}dx}{\left(\int_{\mathbb{R}^{N}} |u_{\epsilon}|^{2_{s}^{\ast}}dx\right)^{\frac{q\gamma_{q,s}-2}{2_{s}^{\ast}-2}}\|u_{\epsilon}\|^{q(1-\gamma_{q,s})}_{L^{2}(\mathbb{R}^{N})}}\right].
\end{align}
By the estimates in \eqref{int6}, \eqref{int7}, \eqref{int9} and \eqref{int12}, we can infer that there exist $C_{1}, C_{2}, C_{3}>0$ (depending on $N,q$) such that

 \begin{align}\label{U26}
 \left(||u_{\epsilon}||^{2}_{D_{s}(\mathbb{R}^{N})}\right)^{\frac{2_{s}^{\ast}-q\gamma_{q,s}}{2_{s}^{\ast}-2}}\geq C_{1} \ \text{and}\ C_{2}\leq\left(\int_{\mathbb{R}^{N}} |u_{\epsilon}|^{2_{s}^{\ast}}dx\right)^{\frac{q\gamma_{q,s}-2}{2_{s}^{\ast}-2}}\leq \frac{1}{C_{2}}
 \end{align}
 and
 \begin{align}\label{U27}
&\frac{\int_{\mathbb{R}^{N}} |u_{\epsilon}|^{q}dx}{\|u_{\epsilon}\|^{q(1-q\gamma_{q,s})}_{L^{2}(\mathbb{R}^{N})}}
 \leq \left\{
\begin{aligned}
&C\epsilon^{N-\frac{N-2s}{2}q-q(1-\gamma_{q,s})},  &\text{if}\ &N>4s, \\
&C\epsilon^{N-\frac{N-2s}{2}q-q(1-\gamma_{q,s})}|\ln \epsilon|^{\frac{q(\gamma_{q,s}-1)}{2}}, &\text{if}\ &N=4s, \\
&C\epsilon^{N-\frac{N-2s}{2}q-\frac{N-2s}{2}q(1-\gamma_{q,s})}, &\text{if} \  &\frac{q}{q-1}2s<N<4s,\\
&C\epsilon^{\frac{N}{2}-\frac{N-2s}{2}q(1-\gamma_{q,s})}|\ln \epsilon| , &\text{if} \   &N=\frac{q}{q-1}2s,\\
&C\epsilon^{\frac{N-2s}{2}q-\frac{N-2s}{2}q(1-\gamma_{q,s})}, &\text{if} \   &2s<N<\frac{q}{q-1}2s.
\end{aligned}
\right.
\end{align}
Next, we claim that
\begin{align*}
e^{(2_{s}^{\ast}-2)st_{v_{\epsilon}, \mu}}\geq  C\frac{\|u_{\epsilon}\|^{2_{s}^{\ast}-2}_{L^{2}(\mathbb{R}^{N})}}{a^{2_{s}^{\ast}-2}},
\end{align*}
under suitable conditions.

\vspace{1mm}

{\bf Case 1:} $N>4s$. Since  $\overline{p}<q<2_{s}^{\ast}$, we can deduce that
\begin{align}\label{U30}
N-\frac{N-2s}{2}q-q(1-\gamma_{q,s})<0.
\end{align}
Indeed, since $\overline{p}<q<2_{s}^{\ast}$, we have $4s/N<q-2<4s/(N-2s)$, so $$N-\frac{N-2s}{2}q-q(1-\gamma_{q,s})=N-\frac{N-2s}{2}(q-2)-(N-2s)-(q-2)-2+\frac{N(q-2)}{2s}:=f(q-2),$$ it is easy to deduce that $f(q-2)$ is strictly increasing about $q-2$, since $f(\frac{4s}{N-2s})=0$, thus we obtain $$N-\frac{N-2s}{2}q-q(1-\gamma_{q,s})<0.$$ So we can not get
\begin{align*}
e^{(2_{s}^{\ast}-2)st_{v_{\epsilon}, \mu}}\geq C\frac{\|u_{\epsilon}\|^{2_{s}^{\ast}-2}_{L^{2}(\mathbb{R}^{N})}}{a^{2_{s}^{\ast}-2}} \left[C_{1}-\mu\gamma_{q,s}a^{q(1-\gamma_{q,s})}\frac{C_{3}}{C_{2}}o_{\epsilon}(1)\right]\geq  C\frac{\|u_{\epsilon}\|^{2_{s}^{\ast}-2}_{L^{2}(\mathbb{R}^{N})}}{a^{2_{s}^{\ast}-2}}
\end{align*}
for a positive constant $C=C(N,q,\mu,a)>0$  for every $\epsilon\in (0,\epsilon_{0})$ with $\epsilon_{0}$ sufficiently small.
Thus, we have to give a more precise estimate, let us recall the inequality about $e^{(2_{s}^{\ast}-2)st_{v_{\epsilon}, \mu}}$ in \eqref{U25}, by  well-known interpolation inequality, we have
\begin{align}\label{U28}
\frac{\int_{\mathbb{R}^{N}} |u_{\epsilon}|^{q}dx}{\left(\int_{\mathbb{R}^{N}} |u_{\epsilon}|^{2_{s}^{\ast}}dx\right)^{\frac{q\gamma_{q,s}-2}{2_{s}^{\ast}-2}}\|u_{\epsilon}\|^{q(1-\gamma_{q,s})}_{L^{2}(\mathbb{R}^{N})}}&\leq\frac{\left(\int_{\mathbb{R}^{N}} |u_{\epsilon}|^{2_{s}^{\ast}}dx\right)^{\frac{q-2}{2_{s}^{\ast}-2}}\left(\int_{\mathbb{R}^{N}} |u_{\epsilon}|^{2}dx\right)^{\frac{2_{s}^{\ast}-q}{2_{s}^{\ast}-2}}}{\left(\int_{\mathbb{R}^{N}} |u_{\epsilon}|^{2_{s}^{\ast}}dx\right)^{\frac{q\gamma_{q,s}-2}{2_{s}^{\ast}-2}}\|u_{\epsilon}\|^{q(1-\gamma_{q,s})}_{L^{2}(\mathbb{R}^{N})}}\\\nonumber
&\leq \left(\int_{\mathbb{R}^{N}} |u_{\epsilon}|^{2_{s}^{\ast}}dx\right)^{\frac{q(1-\gamma_{q,s})}{2_{s}^{\ast}-2}}=\left(\left(\int_{\mathbb{R}^{N}} |u_{\epsilon}|^{2_{s}^{\ast}}dx\right)^{\frac{2}{2_{s}^{\ast}}}\right)^{\frac{2_{s}^{\ast}-q\gamma_{q,s}}{2_{s}^{\ast}-2}}.
\end{align}
Therefore,  by \eqref{U25} and \eqref{U28},  we have
\begin{align}\label{U29}
&e^{(2_{s}^{\ast}-2)st_{v_{\epsilon}, \mu}}
=\frac{||v_{\epsilon}||^{2}_{D_{s}(\mathbb{R}^{N})}}{\int_{\mathbb{R}^{N}} |v_{\epsilon}|^{2_{s}^{\ast}}dx}-\mu\gamma_{q,s}e^{(q\gamma_{q,s}-2)st_{v_{\epsilon}, \mu}}\frac{\int_{\mathbb{R}^{N}} |v_{\epsilon}|^{q}dx}{\int_{\mathbb{R}^{N}} |v_{\epsilon}|^{2_{s}^{\ast}}dx}\\\nonumber
&\geq\frac{\|u_{\epsilon}\|^{2_{s}^{\ast}-2}_{L^{2}(\mathbb{R}^{N})}}{a^{2_{s}^{\ast}-2}} \frac{\left(||u_{\epsilon}||^{2}_{D_{s}(\mathbb{R}^{N})}\right)^{\frac{q\gamma_{q,s}-2}{2_{s}^{\ast}-2}}}{\int_{\mathbb{R}^{N}} |u_{\epsilon}|^{2_{s}^{\ast}}dx}\left[\left(||u_{\epsilon}||^{2}_{D_{s}(\mathbb{R}^{N})}\right)^{\frac{2_{s}^{\ast}-q\gamma_{q,s}}{2_{s}^{\ast}-2}}-\mu\gamma_{q,s}a^{q(1-\gamma_{q,s})}\left(\left(\int_{\mathbb{R}^{N}} |u_{\epsilon}|^{2_{s}^{\ast}}dx\right)^{\frac{2}{2_{s}^{\ast}}}\right)^{\frac{2_{s}^{\ast}-q\gamma_{q,s}}{2_{s}^{\ast}-2}}\right].
\end{align}
Thus, if the right hand of above is positive provided that $$\mu\gamma_{q,s}a^{q(1-\gamma_{q,s})}<\left(\frac{||u_{\epsilon}||^{2}_{D_{s}(\mathbb{R}^{N})}}{\left(\int_{\mathbb{R}^{N}} |u_{\epsilon}|^{2_{s}^{\ast}}dx\right)^{\frac{2}{2_{s}^{\ast}}}}\right)^{{\frac{2_{s}^{\ast}-q\gamma_{q,s}}{2_{s}^{\ast}-2}}}=S^{\frac{N}{4s}q(1-\gamma_{q,s})}_{s}+O(\epsilon^{N-2s})$$
 Thus, if $N>4s\ \text{and}\ \mu a^{q(1-\gamma_{q,s})}<\frac{S^{\frac{N}{4s}q(1-\gamma_{q,s})}_{s}}{\gamma_{q,s}}$, we have $$e^{(2_{s}^{\ast}-2)st_{v_{\epsilon}, \mu}}\geq \frac{C\|u_{\epsilon}\|^{2_{s}^{\ast}-2}_{L^{2}(\mathbb{R}^{N})}}{a^{2_{s}^{\ast}-2}} .$$

{\bf Case 2:}  $N=4s$.  Then we have $3<q<4$ and $|\ln \epsilon|\backsim \frac{1}{\epsilon}$ as $\epsilon\rightarrow 0.$ Thus   $$\epsilon^{N-\frac{N-2s}{2}q-q(1-\gamma_{q,s})}|\ln \epsilon|^{\frac{q(\gamma_{q,s}-1)}{2}}=\epsilon^{(4-q)(s-1)}|\ln \epsilon|^{q-4}\rightarrow 0\  \text{as} \ \epsilon\rightarrow 0.$$
Furthermore, $$\frac{\int_{\mathbb{R}^{N}} |u_{\epsilon}|^{q}dx}{\|u_{\epsilon}\|^{q(1-q\gamma_{q,s})}_{L^{2}(\mathbb{R}^{N})}}\leq C\epsilon^{N-\frac{N-2s}{2}q-q(1-\gamma_{q,s})}|\ln \epsilon|^{\frac{q(\gamma_{q,s}-1)}{2}}=o_{\epsilon}(1).$$
So, we have
\begin{align*}
e^{(2_{s}^{\ast}-2)st_{v_{\epsilon}, \mu}}\geq C\frac{\|u_{\epsilon}\|^{2_{s}^{\ast}-2}_{L^{2}(\mathbb{R}^{N})}}{a^{2_{s}^{\ast}-2}} \left[C_{1}-\mu\gamma_{q,s}a^{q(1-\gamma_{q,s})}\frac{C_{3}}{C_{2}}o_{\epsilon}(1)\right]\geq  C\frac{\|u_{\epsilon}\|^{2_{s}^{\ast}-2}_{L^{2}(\mathbb{R}^{N})}}{a^{2_{s}^{\ast}-2}}.
\end{align*}

{\bf Case 3:} $\frac{q}{q-1}2s<N<4s$. By the same arguments as \eqref{U30}, we have $$N-\frac{N-2s}{2}q-\frac{N-2s}{2}q(1-\gamma_{q,s})>0.$$  Thus, $$\epsilon^{N-\frac{N-2s}{2}q-\frac{N-2s}{2}q(1-\gamma_{q,s})} \rightarrow 0 \  \text{as} \ \epsilon\rightarrow 0.$$
Therefore, $$\frac{\int_{\mathbb{R}^{N}} |u_{\epsilon}|^{q}dx}{\|u_{\epsilon}\|^{q(1-q\gamma_{q,s})}_{L^{2}(\mathbb{R}^{N})}}\leq C\epsilon^{N-\frac{N-2s}{2}q-\frac{N-2s}{2}q(1-\gamma_{q,s})}=o_{\epsilon}(1).$$
So, we have
\begin{align*}
e^{(2_{s}^{\ast}-2)st_{v_{\epsilon}, \mu}}\geq C\frac{\|u_{\epsilon}\|^{2_{s}^{\ast}-2}_{L^{2}(\mathbb{R}^{N})}}{a^{2_{s}^{\ast}-2}} \left[C_{1}-\mu\gamma_{q,s}a^{q(1-\gamma_{q,s})}\frac{C_{3}}{C_{2}}o_{\epsilon}(1)\right]\geq  C\frac{\|u_{\epsilon}\|^{2_{s}^{\ast}-2}_{L^{2}(\mathbb{R}^{N})}}{a^{2_{s}^{\ast}-2}}.
\end{align*}
{\bf Case 4:} $N=\frac{q}{q-1}2s$. By the similar arguments as {\bf Case 1}, we get $$C\epsilon^{\frac{N}{2}-\frac{N-2s}{2}q(1-\gamma_{q,s})}|\ln \epsilon|\rightarrow +\infty  \  \text{as} \ \epsilon\rightarrow 0.$$ Thus, by the same argument as {\bf Case 1}, we know that if $N=\frac{q}{q-1}2s\ \text{and}\ \mu a^{q(1-\gamma_{q,s})}<\frac{S^{\frac{N}{4s}q(1-\gamma_{q,s})}_{s}}{\gamma_{q,s}}$, then we have $$e^{(2_{s}^{\ast}-2)st_{v_{\epsilon}, \mu}}\geq \frac{C\|u_{\epsilon}\|^{2_{s}^{\ast}-2}_{L^{2}(\mathbb{R}^{N})}}{a^{2_{s}^{\ast}-2}} .$$
{\bf Case 5:} $2s<N<\frac{q}{q-1}2s$. It is easy to see that $$\frac{\int_{\mathbb{R}^{N}} |u_{\epsilon}|^{q}dx}{\|u_{\epsilon}\|^{q(1-q\gamma_{q,s})}_{L^{2}(\mathbb{R}^{N})}}\leq C\epsilon^{\frac{N-2s}{2}q-\frac{N-2s}{2}q(1-\gamma_{q,s})}=o_{\epsilon}(1).$$
Then we have
\begin{align*}
e^{(2_{s}^{\ast}-2)st_{v_{\epsilon}, \mu}}\geq C\frac{\|u_{\epsilon}\|^{2_{s}^{\ast}-2}_{L^{2}(\mathbb{R}^{N})}}{a^{2_{s}^{\ast}-2}} \left[C_{1}-\mu\gamma_{q,s}a^{q(1-\gamma_{q,s})}\frac{C_{3}}{C_{2}}o_{\epsilon}(1)\right]\geq  C\frac{\|u_{\epsilon}\|^{2_{s}^{\ast}-2}_{L^{2}(\mathbb{R}^{N})}}{a^{2_{s}^{\ast}-2}}.
\end{align*}

{\bf Step 2.} Estimate on $\max_{t\in \mathbb{R}}\Psi^{\mu}_{v_{\epsilon}}(t)$.
 \begin{align*}
 &\max_{t\in \mathbb{R}}\Psi^{\mu}_{v_{\epsilon}}(t)=\Psi^{\mu}_{v_{\epsilon}}(t_{v_{\epsilon},\mu})=\Psi^{0}_{v_{\epsilon}}(t_{v_{\epsilon},\mu})-\mu\frac{e^{q \gamma_{q,s} st_{v_{\epsilon},\mu}}}{q}\int_{\mathbb{R}^{N}} |v_{\epsilon}|^{q}dx\\
 &\leq \sup_{\mathbb{R}}\Psi^{0}_{v_{\epsilon}}-\frac{\mu C}{q}\frac{\|u_{\epsilon}\|^{q\gamma_{q,s}}_{L^{2}(\mathbb{R}^{N})}}{a^{q\gamma_{q,s}}}\frac{a^{q}}{\|u_{\epsilon}\|^{q}_{L^{2}(\mathbb{R}^{N})}} \int_{\mathbb{R}^{N}} |u_{\epsilon}|^{q}dx\\
 &\leq \frac{s}{N}S^{\frac{N}{2s}}_{s}+O(\epsilon^{N-2s})-C\frac{\mu}{q}\gamma_{q,s}a^{q(1-\gamma_{q,s})}\frac{\int_{\mathbb{R}^{N}} |u_{\epsilon}|^{q}dx}{\|u_{\epsilon}\|^{q(1-q\gamma_{q,s})}_{L^{2}(\mathbb{R}^{N})}}.
 \end{align*}
 By \eqref{U27}, we know that $$\max_{t\in \mathbb{R}}\Psi^{\mu}_{v_{\epsilon}}(t)\leq \frac{s}{N}S^{\frac{N}{2s}}_{s},$$ for $\epsilon$ small enough. This completes the proof.
\end{proof}
\section{Proof of Theorem \ref{TH1}}\label{sec5}

 Let $\{v_{n}\}$  be a minimizing sequence for $\inf_{A_{R_{0}}}E_{\mu}(u)$. By Lemma \ref{Lem4}, for every $n$ we can take $t_{v_{n}}\star v_{n}\in \mathcal{P}^{+}_{a,\mu}$  such that $||t_{v_{n}}\star v_{n}||_{D_{s}(\mathbb{R}^{N})}\leq R_{0}$  and $$E_{\mu}(u)(t_{v_{n}}\star v_{n})=\min\{E_{\mu}(t\star v_{n}):t\in \mathbb{R}\ \text{and}\ ||t\star v_{n}||_{D_{s}(\mathbb{R}^{N})}< R_{0} \}\leq E_{\mu}(v_{n}).$$Thus, we obtain a new minimizing sequence $\{w_{n}=t_{v_{n}}\star v_{n}\}$ with $w_{n}\in S^{r}_{a}\cap\mathcal{P}^{+}_{a,\mu} $ radially decreasing for every $n$. By Lemma \ref{Lem7}, we have $||w_{n}||_{D_{s}(\mathbb{R}^{N})}<R_{0}-\rho$ for every $n$ and hence by Ekeland's variational principle in a standard way, we know the existence of a new minimizing sequence for $\{u_{n}\}\subset A_{R_{0}}$ for $m(a,\mu)$ with $\|u_{n}-w_{n}\|\rightarrow 0$ as $n\rightarrow +\infty $, which is also a Palais-Smale sequence for $E_{\mu}$ on $S_{a}$. By the boundedness of $\{w_{n}\}$, $\|u_{n}-w_{n}\|\rightarrow 0$,  Br\'ezis-Lieb lemma and Sobolev embedding theorem, we have $$\|u_{n}\|^{2}_{D_{s}(\mathbb{R}^{N})}=\|u_{n}-w_{n}\|^{2}_{D_{s}(\mathbb{R}^{N})}+\|w_{n}\|^{2}_{D_{s}(\mathbb{R}^{N})})+o_{n}(1)=\|w_{n}\|^{2}_{D_{s}(\mathbb{R}^{N})})+o_{n}(1),$$
$$\int_{\mathbb{R}^{N}}|u_{n}|^{p}dx=\int_{\mathbb{R}^{N}}|u_{n}-w_{n}|^{p}dx+\int_{\mathbb{R}^{N}}|w_{n}|^{p}+o_{n}(1)=\int_{\mathbb{R}^{N}}|w_{n}|^{p}+o_{n}(1), \ \text{for}\ \ p\in [2,2_{s}^{\ast}].$$
Thus, $$P_{\mu}(u_{n})=P_{\mu}(u_{n})+o_{n}(1)\rightarrow \ \text{as}\ n\rightarrow +\infty.$$
Therefore,  one of the alternatives in  Proposition \ref{LM5} holds. We prove that the second alternative in  Proposition \ref{LM5} occurs. Assume by contradiction that there exists a sequence $u_{n}\rightharpoonup u$ weakly in $H^{s}(\mathbb{R}^{N})$ but not strongly , where $u\not\equiv0$ is a solution of \eqref{int1} for some $\lambda<0$, and $$E_{\mu}(u)\leq m(a,\mu)- \frac{s}{N}S^{\frac{N}{2s}}_{s}.$$ Since $u$ is a solution  of\eqref{int1}, we have $P_{\mu}(u)=0,$ which implies that $$||u||^{2}_{D_{s}(\mathbb{R}^{N})}=\mu\gamma_{q,s}\int_{\mathbb{R}^{N}} |u|^{q}dx+\int_{\mathbb{R}^{N}} |u|^{2_{s}^{\ast}}dx. $$
Therefore
\begin{align*}
m(a,\mu)&\geq E_{\mu}(u)+ \frac{s}{N}S^{\frac{N}{2s}}_{s}= \frac{s}{N}S^{\frac{N}{2s}}_{s}+\frac{s}{N}||u||^{2}_{D_{s}(\mathbb{R}^{N})}
-\frac{\mu}{q}\left(1-\frac{q\gamma_{q,s}}{2_{s}^{\ast}}\right)\int_{\mathbb{R}^{N}}|u|^{q}dx\\
&\geq \frac{s}{N}S^{\frac{N}{2s}}_{s}+\frac{s}{N}||u||^{2}_{D_{s}(\mathbb{R}^{N})}
-\frac{\mu}{q}\left(1-\frac{q\gamma_{q,s}}{2_{s}^{\ast}}\right)C^{q}_{N,q,s}a^{q(1-\gamma_{q,s})}||u||^{q\gamma_{q,s}}_{D_{s}(\mathbb{R}^{N})}.
\end{align*}
Next, we show that the right hand side of above inequality is positive under suitable conditions, then we can get a contradiction with $m(a,\mu)<0$. Let $$\vartheta(t)=\frac{s}{N}t^{2}
-\frac{\mu}{q}\left(1-\frac{q\gamma_{q,s}}{2_{s}^{\ast}}\right)C^{q}_{N,q,s}a^{q(1-\gamma_{q,s})}t^{q\gamma_{q,s}}.$$
Then it is easy to see that the function $\vartheta(t)$ has a unique minimum point $\overline{t}$ and $$\vartheta(\overline{t})=-\frac{2-q\gamma_{q,s}}{q}\left[\frac{N\gamma_{q,s}}{s}\right]^{\frac{q\gamma_{q,s}}{2-q\gamma_{q,s}}}\left[\frac{2_{s}^{\star}-q\gamma_{q,s}}{22_{s}^{\star}}C^{q}_{N,q,s}\right]^{\frac{2}{2-q\gamma_{q,s}}}[\mu a^{q(1-\gamma_{q,s})}]^{\frac{2}{2-q\gamma_{q,s}}}<0.$$
If $$\vartheta(\overline{t})>-\frac{s}{N}S^{\frac{N}{2s}}_{s},$$ which yields that $$\mu a^{q(1-\gamma_{q,s})}<\frac{22_{s}^{\star}}{(2_{s}^{\star}-q\gamma_{q,s})C^{q}_{N,q,s}}\left(\frac{Nq\gamma^{2}_{q,s}S^{\frac{N}{2s}}_{s}}{(2-q\gamma_{q,s})s}\right)^{\frac{2-q\gamma_{q,s}}{2}},$$
then we have
\begin{align*}
m(a,\mu)\geq  \frac{s}{N}S^{\frac{N}{2s}}_{s}+\vartheta(||u||_{D_{s}(\mathbb{R}^{N})})\geq \frac{s}{N}S^{\frac{N}{2s}}_{s}+\vartheta(\overline{t})>0,
\end{align*}
which contradicts the fact that $m(a,\mu)<0$. Thus  $u_{n}\rightarrow u$ strongly in $H^{s}(\mathbb{R}^{N}),$ $E_{\mu}(u)=m(a,\mu)$ and $u$ solves \eqref{int1}--\eqref{in2} for some $\lambda<0.$ It remains to show that any ground state is a local minimizer for $E_{\mu}$ on $A_{R_{0}}$.  Since $E_{\mu}(u)=m(a,\mu)$, then $u\in \mathcal{P}_{a,\mu}$ and $E_{\mu}(u)<0$, so by Lemma \ref{Lem4} we have that $u\in \mathcal{P}^{+}_{a,\mu} \subset  A_{R_{0}}$ and $$E_{\mu}(u)=m(a,\mu)=\inf_{ A_{R_{0}}}E_{\mu}(u)\ \text{and}\ ||u||_{D_{s}(\mathbb{R}^{N})}<R_{0}.$$
Therefore, the proof of Theorem \ref{TH1} is complete.

\section{Proof of Theorems \ref{TH2}--\ref{TH3}}\label{sec6}
We first list some well-known results, which will be used to prove Theorems \ref{TH2}--\ref{TH3}. For this purpose, we give the following definition.
\begin{definition}{\rm (see \cite[Definition 3.1]{NGB})}
Let B be a closed subset of X. We shall say that a class $\mathcal{F}$ of compact subsets of X is a homotopy-stable family with boundary B provided that
 \begin{enumerate}
\item[$(a)$] every set in $\mathcal{F}$ contains B.
\item[$(b)$] for any set A in $\mathcal{F}$ and any $\eta\in([0,1]\times X;X)$ satisfying $\eta(t,x)=x$ for all $(t,x)\in ({0}\times X)\cup ([0,1]\times B)$, we have $\eta({1}\times A)\in \mathcal{F}.$
 \end{enumerate}
\end{definition}
\begin{theorem}\label{aa}{\rm (see \cite[Theorem 3.2]{NGB})}
Let $\varphi$ be a $C^{1}$ function on a complete connected $C^{1}$-Finsler manifold X (without boundary) and consider a homotopy-stable family $\mathcal{F}$ of compact subsets of X with a closed boundary B. Set $c=c(\varphi,\mathcal{F})=\inf\limits_{A\in \mathcal{F}}\max\limits_{x\in A}\varphi(x)$ and suppose that $$\sup\varphi(B)<c.$$
Then, for any sequence of sets $(A_{n})_{n}$ in $\mathcal{F}$ such that $\lim\limits_{n}\sup\limits_{A_{n}}\varphi=c,$ there exists a sequence $(x_{n})_{n}$ in X such that

${\rm (i)}$\ $\lim\limits_{n}\varphi(x_{n})=c \ \ ${\rm (ii)}$\ \lim\limits_{n}\|d\varphi(x_{n})\|=0 \ \ ${\rm (iii)}$\ \lim\limits_{n}{\rm dist}(x_{n}, A_{n})=0.$

\noindent Moreover, if $d\varphi$ is uniformly continuous, then $x_{n}$ can be chosen to be in $A_{n}$ for each n.
\end{theorem}

Now we are in a position to prove Theorems \ref{TH2}--\ref{TH3}.

\vspace{1mm}

{\bf Case 1: $L^{2}$-critical perturbation,} i.e.,  $q=\overline{p}$.   Let $k>0$ be defined by Lemma \ref{Lem11}, we follow the ideas  introduced in \cite{Jeanjean} and consider the following functional $\overline{E}_{\mu}(s,\mu): \mathbb{R}\times H^{s}(\mathbb{R}^{N})\rightarrow \mathbb{R}$:
\begin{align}\label{U32}
\widetilde{E}_{\mu}(t,\mu)=E_{\mu}(t\star u)=\left[\frac{1}{2}||u||^{2}_{D_{s}(\mathbb{R}^{N})}-\frac{\mu}{\overline{p}}\int_{\mathbb{R}^{N}} |u|^{\overline{p}}dx\right]e^{2st}-\frac{e^{2_{s}^{\ast}st}}{2_{s}^{\ast}}\int_{\mathbb{R}^{N}} |u|^{2_{s}^{\ast}}dx.
\end{align}
It is easy to see that $\overline{E}_{\mu}(t,\mu)$ is of class of $C^{1}$, since $\overline{E}_{\mu}(t,\mu)$ is invariant under rotations applied to $u$, a Palais-Smale sequence  for $\overline{E}_{\mu}(t,\mu)|_{\mathbb{R}\times S^{r}_{a}}$  is a Palais-Smale sequence  for $\overline{E}_{\mu}(t,\mu)|_{\mathbb{R}\times S_{a}}$. Let $E^{c}$ be the closed sublevel set $\{u\in S_{a}: E_{\mu}\leq c\},$ we introduce the minimax class
\begin{align}\label{U33}
\Gamma:=\{\gamma=(\alpha,\beta)\in C([0,1],\mathbb{R}\times S^{r}_{a}): \gamma(0)\in (0,\overline{A_{k}}),\gamma(1)\in (0,E^{0}) \}
\end{align}
with associated minimax level $$\sigma(a,\mu):=\inf_{\gamma\in \Gamma}\max_{(t,u)\in \gamma([0,1])}\widetilde{E}_{\mu}(t,\mu).$$
 Since $||t\star u||^{2}_{D_{s}(\mathbb{R}^{N})}\rightarrow 0^{+}$ as $t\rightarrow -\infty$ and $E_{\mu}(t\star u)\rightarrow -\infty$ as $t\rightarrow +\infty$. Let $u\in S^{r}_{a}$. There exist $t_{0}\ll -1$ and $t_{1}\gg 1$ such that
\begin{align}\label{U34}
\gamma_{u}: \tau\in [0,1]\rightarrow (0,((1- \tau)t_{0}+\tau t_{1})\star u)\in \mathbb{R}\times S^{r}_{a}
\end{align}
is a path in $\Gamma$. Then $\sigma(a,\mu)$ is a real number.  For any $\gamma=(\alpha,\beta)\in \Gamma$, let us consider the function $$P_{\gamma}:\tau\in [0,1]\rightarrow P(\alpha(\tau)\star \beta(\tau))\in \mathbb{R}.$$ We have $P_{\gamma}(0)=P(\beta(0))>0$, by Lemmas \ref{Lem10} and  \ref{Lem11}. Since $\Psi_{\beta(1)}(t)>0$ for every $t\in (-\infty, t_{\beta(1)})$ and $\Psi_{\beta(1)}(0)=E_{\mu}(\beta(1))\leq0,$ we have that $t_{\beta(1)}<0.$ Thus, by Lemma \ref{Lem8}, we have  $P_{\gamma}(1)=P(\beta(1))<0.$ Moreover, the map $\tau:\alpha(\tau)\star \beta (\tau)$ is continuous from $[0,1]$ to $H^{s}(\mathbb{R}^{N})$, hence we deduce that there exists  $\tau_{\gamma}\in (0,1)$ such that $P_{\gamma}(\tau_{\gamma})=0$, namely $\alpha(\tau_{\gamma})\star \beta(\tau_{\gamma})\in \mathcal{P}_{a,\mu} $, this implies that $$\max_{\gamma([0,1])}\widetilde{E}_{\mu}\geq \widetilde{E}_{\mu}(\gamma(\tau_{\gamma}))=E_{\mu}(\alpha(\tau_{\gamma})\star \beta(\tau_{\gamma}))\geq \inf_{\mathcal{P}_{a,\mu}\cap S^{r}_{a}}E_{\mu}=m_{r}(a,\mu).$$ Consequently, $\sigma(a,\mu)\geq m_{r}(a,\mu)$. On the other hand, if $u\in \mathcal{P}_{a,\mu} \cap S^{r}_{a}$, then $\gamma_{u}$ defined in \eqref{U34} ia s path in $\Gamma$ with $$E_{\mu}(u)=\max_{\gamma_{u}([0,1])}\widetilde{E}_{\mu}\geq \sigma(a,\mu),$$ which implies that $$m_{r}(a,\mu)\geq \sigma(a,\mu).$$ Combining  this  with Lemmas \ref{Lem10}--\ref{Lem11}, we have that $$\sigma(a,\mu)=m_{r}(a,\mu)>\sup_{(\overline{A_{k}}\cup E^{0})\cap S^{r}_{a}}E_{\mu}=\sup_{((0,\overline{A_{k}})\cup (0,E^{0}))\cap (\mathbb{R}\times S^{r}_{a})}\widetilde{E}_{\mu}.$$
By Theorem \ref{aa}, we know that $\{\gamma([0,1]): \gamma\in \Gamma\}$ is a homotopy stable family of compact subsets of $\mathbb{R}\times S^{r}_{a}$ with closed boundary $(0, \overline{A}_{k})\cup (0,E^{0})$ and the superlevel set $\{\widetilde{E}\geq \sigma(a,\mu)\}$ is a dual set for $\Gamma$. By Theorem \ref{aa} we can taking any minimizing sequence $\{\gamma_{n}=(\alpha_{n},\beta_{n})\}\subset \Gamma_{n}$ for $\sigma(a,\mu)$ with the property  that $\alpha_{n}=0$ and $\beta_{n}(\tau)\geq 0$ a.e in $\mathbb{R}^{N}$, there exists a Palais-Smale sequence $\{(t_{n},w_{n})\}\subset \mathbb{R}\times S^{r}_{a}$ for $\widetilde{E}|_{\mathbb{R}\times S^{r}_{a}}$ at level $\sigma(a,\mu)$, that is
\begin{align}\label{U35}
\partial_{t}\widetilde{E}_{\mu} (t_{n},w_{n}) \rightarrow 0\  \text{and}\ \ \|\partial_{u}\widetilde{E}_{\mu} (t_{n},w_{n})\|\rightarrow 0 \ \text{as}\ \ n\rightarrow +\infty.
\end{align}
with the additional property that
\begin{align}\label{U36}
|t_{n}|+{\rm dist}_{H^{s}} (w_{n},\beta_{n}([0,1]))\rightarrow 0 \ \text{as}\ \ n\rightarrow +\infty.
\end{align}
By the definition of $\widetilde{E}_{\mu} (t_{n},w_{n})$ in \eqref{U32}, from \eqref{U35} we know that $P (t_{n},w_{n}) \rightarrow 0$,
that is
\begin{align}\label{U37}
dE_{\mu} (t_{n}\star w_{n})[t_{n}\star \varphi]=o(1) \|\varphi\|=o(1) \|t_{n}\star \varphi\| \ \text{as}\ \ n\rightarrow +\infty\ \text{for every} \ \varphi\in T_{w_{n}}S^{r}_{a}.
\end{align}
Let $u_{n}=t_{n}\star w_{n}$, by \eqref{U37}, we know that $\{u_{n}\}$  is a Palais-Smale sequence for $E_{\mu}|_{S^{r}_{a}}$ at the level $\sigma(a,\mu)=m_{r}(a,\mu)$ and $P(u_{n})\rightarrow 0$. Thus, by Lemmas \ref{Lem10}--\ref{Lem16}, we obtain that $m_{r}(a,\mu)\in (0, \frac{s}{N}S^{\frac{N}{2s}}_{s})$, so by Proposition \ref{LM5}, one of the alternatives occurs. Assume (i) occurs in Proposition \ref{LM5}, then up to a subsequence $u_{n}\rightharpoonup \widetilde{u}$ weakly in $H^{s}(\mathbb{R}^{N})$ but not strongly, where $\widetilde{u}\not\equiv0$ is a solution of \eqref{int1} for some $\lambda<0$, and $$E_{\mu}(\widetilde{u})\leq m(a,\mu)- \frac{s}{N}S^{\frac{N}{2s}}_{s}<0.$$ Hence by the Pohozaev identity,  $P(\widetilde{u})=0$ holds, which implies that $$||\widetilde{u}||^{2}_{D_{s}(\mathbb{R}^{N})}
-\frac{2\mu}{\overline{p}}\int_{\mathbb{R}^{N}}|\widetilde{u}|^{\overline{p}}dx-\int_{\mathbb{R}^{N}}|\widetilde{u}|^{2_{s}^{\ast}}dx=0.$$
Thus
\begin{align*}
E_{\mu}(u)=\frac{1}{2}||\widetilde{u}||^{2}_{D_{s}(\mathbb{R}^{N})}
-\frac{2\mu}{\overline{p}}\int_{\mathbb{R}^{N}}|\widetilde{u}|^{\overline{p}}dx-\frac{1}{2_{s}^{\ast}}\int_{\mathbb{R}^{N}}|\widetilde{u}|^{2_{s}^{\ast}}dx=\frac{s}{N}\int_{\mathbb{R}^{N}}|\widetilde{u}|^{2_{s}^{\ast}}dx>0,
\end{align*}
which contradicts the fact that $$E_{\mu}(\widetilde{u})\leq m(a,\mu)- \frac{s}{N}S^{\frac{N}{2s}}_{s}<0.$$
Therefore, the alternative (ii) in Proposition \ref{LM5} holds. There exists a  subsequence $u_{n}\rightarrow \widetilde{u}$ strongly in $H^{s}(\mathbb{R}^{N}),$ $E_{\mu}(\widetilde{u})=m(a,\mu)$ and $\widetilde{u}$ solves \eqref{int1}--\eqref{in2} for some $\lambda<0.$ By $\beta_{n}(\tau)\geq 0$ a.e in $\mathbb{R}^{N}$, \eqref{U36} and the convergence implies that $\widetilde{u}\geq 0$, by the strong maximum principle for the fractional Laplacian, see Proposition 2.17 in \cite{LS}, we have $\overline{u}$  is positive. Finally, we prove that $\widetilde{u}$ is a ground state solution. Since any normalized solutions in $\mathcal{P}_{a,\mu}$ and $$E_{\mu}(\overline{u})=m_{r}(a,\mu)=\inf_{\mathcal{P}_{a,\mu}\cap S_{a}}E_{\mu}.$$ It is sufficient to show that $$\inf_{\mathcal{P}_{a,\mu}\cap S_{a}}E_{\mu}=\inf_{\mathcal{P}_{a,\mu}}E_{\mu}=m(a,\mu).$$Assume by contradiction that there is a $u\in \mathcal{P}_{a,\mu}\setminus S^{r}_{a}$  such that $E_{\mu}(u)<\inf_{\mathcal{P}_{a,\mu}\cap S_{a}}E_{\mu}$  and there exists a minimizer $u$, let $v=|u|^{\ast}$ the symmetric decreasing rearrangement of $u$. Then by the properties of symmetric decreasing rearrangement, we have $$||v||^{2}_{D_{s}(\mathbb{R}^{N})}\leq ||u||^{2}_{D_{s}(\mathbb{R}^{N})},\ E_{\mu}(v)\leq E_{\mu}(u)\ \text{and}\ P_{\mu}(v)\leq 0=P_{\mu}(u).$$If $P_{\mu}(v)=0,$ then $P_{\mu}(v)=P_{\mu}(v)=0,$ which is a contradiction with above inequalities.  If $P_{\mu}(v)<0,$ then by Lemma \ref{Lem8}, we know that $t_{v,\mu}<0$, thus $$E_{\mu}(u)\leq E_{\mu}(t_{v,\mu}\star u)=\frac{se^{2_{s}^{\star}st_{v,\mu}}}{N}\int_{\mathbb{R}^{N}}|v|^{2_{s}^{\ast}}dx= \frac{se^{2_{s}^{\star}st_{v,\mu}}}{N}\int_{\mathbb{R}^{N}}|v|^{2_{s}^{\ast}}dx=e^{2_{s}^{\star}st_{v,\mu}}E_{\mu}(u)<E_{\mu}(u),$$  which is a contraction. Thus $$m(a,\mu)=m_{r}(a,\mu)$$ and hence $\widetilde{u}$ is a ground state solution.

\vspace{1mm}

{\bf Case 2: $L^{2}$-supercritical perturbation,} i.e.,  $2+4s/N<q<2_{s}^{\ast}$.   Proceeding exactly as in the case $q=\overline{p}$, we obtain a Palais-Smale sequence $\{u_{n}\}\subset S^{r}_{a}$ for $E_{\mu}|_{S_{a}}$ at the level $\sigma(a,\mu)=m_{r}(a,\mu)$ and $P(u_{n})\rightarrow 0$. Thus, by Lemma \ref{Lem17}, we obtain that $m_{r}(a,\mu)\in (0, \frac{s}{N}S^{\frac{N}{2s}}_{s})$, so by Proposition \ref{LM5}, one of the alternatives occurs. Assume (i) occurs in Proposition \ref{LM5}, then up to a subsequence $u_{n}\rightharpoonup \widetilde{u}$ weakly in $H^{s}(\mathbb{R}^{N})$ but not strongly, where $\widetilde{u}\not\equiv0$ is a solution of \eqref{int1} for some $\lambda<0$, and $$E_{\mu}(\widetilde{u})\leq m(a,\mu)- \frac{s}{N}S^{\frac{N}{2s}}_{s}<0,$$ hence by the Pohozaev identity  $P(\widetilde{u})=0$ holds, which implies that $$||\widetilde{u}||^{2}_{D_{s}(\mathbb{R}^{N})}-\mu\gamma_{q,s}\int_{\mathbb{R}^{N}} |\widetilde{u}|^{q}dx-\int_{\mathbb{R}^{N}} |\widetilde{u}|^{2_{s}^{\ast}}dx=0,$$
thus, by $q\gamma_{q,s}>2$, we have
\begin{align*}
E_{\mu}(u)&=\frac{1}{2}||\widetilde{u}||^{2}_{D_{s}(\mathbb{R}^{N})}
-\frac{\mu}{q}\int_{\mathbb{R}^{N}}|\widetilde{u}|^{q}dx-\frac{1}{2_{s}^{\ast}}\int_{\mathbb{R}^{N}}|\widetilde{u}|^{2_{s}^{\ast}}dx\\
&=\frac{\mu}{q}\left(\frac{q\gamma_{q,s}}{2}-1\right)\int_{\mathbb{R}^{N}}|\widetilde{u}|^{q}dx+\frac{s}{N}\int_{\mathbb{R}^{N}}|\widetilde{u}|^{2_{s}^{\ast}}dx>0,
\end{align*}
which contradicts the fact that $$E_{\mu}(\widetilde{u})\leq m(a,\mu)- \frac{s}{N}S^{\frac{N}{2s}}_{s}<0.$$Therefore, the alternative (ii) in Proposition \ref{LM5} holds. There exists a  subsequence $u_{n}\rightarrow \widetilde{u}$ strongly in $H^{s}(\mathbb{R}^{N}),$ $E_{\mu}(\widetilde{u})=m(a,\mu)$ and $\widetilde{u}$ solves \eqref{int1}--\eqref{in2} for some $\lambda<0.$ By $\beta_{n}(\tau)\geq 0$ a.e in $\mathbb{R}^{N}$, \eqref{U36} and the convergence implies that $\widetilde{u}\geq 0$, by the strong maximum principle for fractional Laplacian
(see Proposition 2.17 in \cite{LS}), we have $\overline{u}$  is positive. The next arguments are the same as {\bf case 1}.
This completes the proof.
\section{Proof of Theorem \ref{Lem1}}\label{Lab5}
\begin{proof}[Proof of Theorem \ref{Lem1}]
If we focus on the case $\mu=0$, then
\begin{align*}
E_{0}(u)&=\frac{1}{2}||u||^{2}_{D_{s}(\mathbb{R}^{N})}
-\frac{1}{2_{s}^{\ast}}\int_{\mathbb{R}^{N}}|u|^{2_{s}^{\ast}}dx
\end{align*}
on $S_{a}$. The associated Pohozaev manifold is $$\mathcal{P}_{a,0}=\bigg\{u\in S_{a}:s||u||^{2}_{D_{s}(\mathbb{R}^{N})}=s\int_{\mathbb{R}^{N}} |u|^{2_{s}^{\ast}}dx\bigg\}=\bigg\{u\in S_{a}:(\Psi^{0}_{u})'(0) =0\bigg\},$$
where \begin{equation*}
 \Psi^{0}_{u}(t)=\frac{e^{2st}}{2}||u||^{2}_{D_{s}(\mathbb{R}^{N})}-\frac{e^{2_{s}^{\ast}st}}{2_{s}^{\ast}}\int_{\mathbb{R}^{N}} |u|^{2_{s}^{\ast}}dx.
\end{equation*}
Recall the decomposition $$\mathcal{P}_{a,0}=\mathcal{P}^{+}_{a,0}\cup \mathcal{P}^{0}_{a,0}\cup \mathcal{P}^{-}_{a,0}.$$
Since \begin{equation*}
 \Psi^{0}_{u}(t)=\frac{e^{2st}}{2}||u||^{2}_{D_{s}(\mathbb{R}^{N})}-\frac{e^{2_{s}^{\ast}st}}{2_{s}^{\ast}}\int_{\mathbb{R}^{N}} u^{2_{s}^{\ast}}dx.
\end{equation*}
It is easy to see that for every $u\in S_{a}$, the function $\Psi^{0}_{u}(t)$ has a unique critical point $t_{u,0}$, which is  a strict maximum point and is given by
\begin{align}\label{U1}
e^{st_{u,0}}=\left(\frac{||u||^{2}_{D_{s}(\mathbb{R}^{N})}}{\int_{\mathbb{R}^{N}} |u|^{2_{s}^{\ast}}dx}\right)^{\frac{1}{2_{s}^{\ast}-2}}.
\end{align}
By the definition of $\mathcal{P}^{+}_{a,0}$, we know that $\mathcal{P}^{+}_{a,0}=\emptyset$. If $u\in \mathcal{P}^{0}_{a,0}$, then $u\in \mathcal{P}_{a,0} $ and $ (\Psi^{\mu}_{u})''(0) =0$, which implies that $$2||u||^{2}_{D_{s}(\mathbb{R}^{N})}=2_{s}^{\ast}\int_{\mathbb{R}^{N}} |u|^{2_{s}^{\ast}}dx=2_{s}^{\ast}||u||^{2}_{D_{s}(\mathbb{R}^{N})}\Rightarrow ||u||_{D_{s}(\mathbb{R}^{N})}=0,$$ which is not possible since $u\in S_{a}$. Then $\mathcal{P}_{a,0}=\mathcal{P}^{-}_{a,0}$.

Next, we show that $\mathcal{P}_{a,0}$ is a smooth manifold of codimension 1 on $S_{a}$. Since $$\mathcal{P}_{a,0}=\bigg\{u\in S_{a}:||u||^{2}_{D_{s}(\mathbb{R}^{N})}=\int_{\mathbb{R}^{N}} |u|^{2_{s}^{\ast}}dx\bigg\},$$ we know that $\mathcal{P}_{a,0}$ is defined by $P_{0}(u)=0$, $G(u)=0$, where $$P_{0}(u)=s||u||^{2}_{D_{s}(\mathbb{R}^{N})}-s\int_{\mathbb{R}^{N}} |u|^{2_{s}^{\ast}}dx \ \ \text{and}\ \ G(u)=\int_{\mathbb{R}^{N}} |u|^{2}dx=a^{2} .$$ Since $P_{0}(u)$ and $G(u)$ are class of $C^{1}$, we only need to check that $d(P_{0}(u),G(u))$: $H^{s}(\mathbb{R}^{N})\rightarrow \mathbb{R}^{2}$ is surjective. If this not true, $dP_{0}(u)$ has to be linearly dependent from $dG(u)$ i.e. there exist a $\nu\in \mathbb{R}$ such that $$2s\int_{\mathbb{R}^{N}} (-\Delta)^{\frac{s}{2}}u(-\Delta)^{\frac{s}{2}}\varphi dx-s2_{s}^{\ast}\int_{\mathbb{R}^{N}} |u|^{2_{s}^{\ast}-2}u\varphi dx=\nu\int_{\mathbb{R}^{N}} u\varphi dx\ \ \text{for every } \ \ \varphi\in H^{s}(\mathbb{R}^{N}),$$ which implies that $$2s(-\Delta)^{2}u=\nu u+2_{s}^{\ast}s u^{2_{s}^{\ast}-1}\ \ \text{in}\ \mathbb{R}^{N}.$$
By the Pohozaev identity for above equation, we know that $$2s||u||^{2}_{D_{s}(\mathbb{R}^{N})}=2_{s}^{\ast}s\int_{\mathbb{R}^{N}} |u|^{2_{s}^{\ast}}dx,$$
that is, $u\in \mathcal{P}^{+}_{a,0}$, a contradiction. Hence $\mathcal{P}_{a,0}$ is a natural constraint.

Indeed, if $u\in \mathcal{P}_{a,0}$ is a critical point of $E_{0}|_{\mathcal{P}_{a,0}}$, then $u$ is a critical point of $E_{0}|_{S_{a}}$.
 Thus, for every $u\in S_{a}$ there exist a unique $t_{u,0}\in \mathbb{R}$ such that $t_{u,0}\star u \in \mathcal{P}_{a,0} $  and $t_{u,0}$ is a strict maximum point of $\Psi^{0}_{u}(t)$, if $u\in \mathcal{P}_{a,0}$, we have that $t_{u,0}=0$ and $$E_{0}(u)=\max_{t\in \mathbb{R}}E_{0}(t\star u)\geq \inf_{v\in S_{a}} \max_{t\in \mathbb{R}}E_{0}(t\star u).$$ On the other hand, if $u\in S_{a},$ then $t_{u,0}\star u \in\mathcal{P}_{a,0} $, so $$\max_{t\in \mathbb{R}}E_{0}(t\star u)=E_{0}(t_{u,0}\star u)\geq \inf_{v\in \mathcal{P}_{a,0}} E_{0}(v).$$ Thus $$\inf_{u\in \mathcal{P}_{a,0}} E_{0}(u)=\inf_{u\in S_{a}} \max_{t\in \mathbb{R}}E_{0}(t\star u).$$
 Now, by \eqref{U1}, we have
 \begin{align*}
 \inf_{u\in \mathcal{P}_{a,0}} E_{0}(u)&=\inf_{u\in S_{a}} \max_{t\in \mathbb{R}}E_{0}(t\star u)\\
 &=\inf_{u\in S_{a}} \left[\frac{1}{2}\left(\frac{||u||^{2}_{D_{s}(\mathbb{R}^{N})}}{\int_{\mathbb{R}^{N}} |u|^{2_{s}^{\ast}}dx}\right)^{\frac{2}{2_{s}^{\ast}-2}}||u||^{2}_{D_{s}(\mathbb{R}^{N})}-\frac{1}{2_{s}^{\ast}}\left(\frac{||u||^{2}_{D_{s}(\mathbb{R}^{N})}}{\int_{\mathbb{R}^{N}} |u|^{2_{s}^{\ast}}dx}\right)^{\frac{2_{s}^{\ast}}{2_{s}^{\ast}-2}}\int_{\mathbb{R}^{N}} |u|^{2_{s}^{\ast}}dx\right]\\
 &=\inf_{u\in S_{a}} \frac{s}{N}\left(\frac{||u||^{2}_{D_{s}(\mathbb{R}^{N})}}{\left(\int_{\mathbb{R}^{N}} |u|^{2_{s}^{\ast}}dx\right)^{\frac{2}{2_{s}^{\ast}}}}\right)^{\frac{2_{s}^{\ast}}{2_{s}^{\ast}-2}}=\inf_{u\in H^{s}(\mathbb{R}^{N})\setminus \{0\}} \frac{s}{N}\left(\frac{||u||^{2}_{D_{s}(\mathbb{R}^{N})}}{\left(\int_{\mathbb{R}^{N}} |u|^{2_{s}^{\ast}}dx\right)^{\frac{2}{2_{s}^{\ast}}}}\right)^{\frac{N}{2s}}.
 \end{align*}
 So it follows that
  $$\inf_{u\in H^{s}(\mathbb{R}^{N})\setminus \{0\}} \frac{s}{N}\left(\frac{||u||^{2}_{D_{s}(\mathbb{R}^{N})}}{\left(\int_{\mathbb{R}^{N}} |u|^{2_{s}^{\ast}}dx\right)^{\frac{2}{2_{s}^{\ast}}}}\right)^{\frac{N}{2s}}=\frac{s}{N}S^{\frac{N}{2s}}_{s}$$
     and the infimum is achieved if and only if by the extremal functions $U_{\epsilon,y}$ defined in \eqref{G1} when $N>4s$ and stay in $L^{2}(\mathbb{R}^{N})$. If $2s<N\leq 4s$, we show that the infimum  of $E_{0}$ on $\mathcal{P}_{a,0}$ is not achieved. Assume by contradiction that there exists a minimizer $u$, let $v=|u|^{\ast}$ the symmetric decreasing rearrangement of $u$. Then by the properties of symmetric decreasing rearrangement, we have $$||v||^{2}_{D_{s}(\mathbb{R}^{N})}\leq ||u||^{2}_{D_{s}(\mathbb{R}^{N})},\ E_{0}(v)\leq E_{0}(u)\ \text{and}\ P_{0}(v)\leq 0=P_{0}(u).$$ If $P_{0}(v)<0,$ then by \eqref{U1}, we know that $t_{v,0}<0$, thus $$E_{0}(u)\leq E_{0}(t_{v,0}\star u)=\frac{se^{2st_{v,0}}}{N}||v||^{2}_{D_{s}(\mathbb{R}^{N})}\leq \frac{se^{2st_{v,0}}}{N}||u||^{2}_{D_{s}(\mathbb{R}^{N})}=e^{2st_{v,0}}E_{0}(u)<E_{0}(u),$$ which is a contradiction. Thus $P_{0}(v)=0\Rightarrow v\in \mathcal{P}_{a,0} $. Since $\mathcal{P}_{a,0}$ is a natural constraint, we obtain
 \begin{align}\label{U2}
 (-\Delta)^{s}v=\lambda v +v^{2_{s}^{\ast}-1}, \ v\geq 0\ \text{in}\ \mathbb{R}^{N},
 \end{align}
 for some $\lambda\in \mathbb{R}$. Since $P_{0}(v)=0$, which implies that $\lambda=0$. By the strong maximum principle, we have $v>0$ in $\mathbb{R}^{N}$. From \cite{JLX}, we know that $v=\alpha U_{\epsilon,0}$ for some $\alpha,\epsilon>0$, this is not possible, since $U_{\epsilon,0}\notin H^{s}(\mathbb{R}^{N})$ for $2s<N\leq 4s.$
The proof is thus complete.
\end{proof}
\section{Proof of Theorem \ref{TH4}}\label{sec7}
In this section, we prove Theorem \ref{TH4}. Before the proof, we give some lemmas.
\begin{lemma}\label{Lem18}
Let $a>0$ , $\mu\geq 0,\ \overline{p}\leq q< 2_{s}^{\ast}$ and  \eqref{equ4} holds. Then $$\inf_{u\in \mathcal{P}_{a,\mu} }E_{\mu}(u)=\inf_{u\in S_{a} }\max_{t\in \mathbb{R}}E_{\mu}(t\star u).$$
\end{lemma}
\begin{proof}
Since $\overline{p}\leq q< 2_{s}^{\ast}$ and $\mu\geq0$, by  Lemma \ref{Lem8} and Lemma \ref{Lem13}, we know that $\mathcal{P}_{a,\mu}=\mathcal{P}^{-}_{a,\mu}$, for every $u\in S_{a}$, there is a unique $t_{u,\mu}\in \mathbb{R}$ such that $t_{u,\mu}\star u \in \mathcal{P}_{a,\mu},$ where $t_{u,\mu}$ is the unique critical point of the function of $\Psi^{\mu}_{u}$ (see Proposition \ref{Lem1} for $\mu=0$). So, if $u\in \mathcal{P}_{a,\mu},$  we have that $t_{u,\mu}=0$ and $$E_{\mu}(u)=\max_{t\in \mathbb{R}}E_{\mu}(t\star u)\geq \inf_{v\in S_{a}}\max_{t\in \mathbb{R}}E_{\mu}(t\star v) .$$
On the other hand, if $u\in S_{a},$ then $t\star u \in \mathcal{P}_{a,\mu}$ and hence $$\max_{t\in \mathbb{R}}E_{\mu}(t\star u)=E_{\mu}(t_{u,\mu}\star u)\geq \inf_{v\in \mathcal{P}_{a,\mu}}E_{\mu}( v).$$
This ends the proof.
\end{proof}
\begin{lemma}\label{Lem19}
Let $a>0$, $\overline{p}\leq q< 2_{s}^{\ast}$,\ $\widetilde{\mu}\geq 0$ satisfy \eqref{equ4} holds. Then the function $\mu\in[0,\widetilde{\mu}]\rightarrow m(a,\mu)\in \mathbb{R}$ is monotone non-increasing.
\end{lemma}
\begin{proof}
Let $0\leq \mu_{1}\leq \mu_{2}\leq \widetilde{\mu}$, by Lemma \ref{Lem18}, we know that
\begin{align*}
m(a,\mu_{2})&=\inf_{u\in S_{a} }\max_{t\in \mathbb{R}}E_{\mu_{2}}(t\star u)=\inf_{u\in S_{a} }E_{\mu_{2}}(t_{u,\mu_{2}}\star u)\\
&=\inf_{u\in S_{a} }\left[E_{\mu_{1}}(t_{u,\mu_{2}}\star u)+(\mu_{1}-\mu_{2})\frac{e^{q\gamma_{q,s}st}}{q}\int_{\mathbb{R}^{N}} |u|^{q}dx\right]\\
&\leq \inf_{u\in S_{a} }\max_{t\in \mathbb{R}}E_{\mu_{1}}(t\star u)=m(a,\mu_{1}).
\end{align*}
As desired.
\end{proof}
\begin{proof}[Proof of Theorem \ref{TH4}]
We divide the proof into two cases.

\smallskip

{\bf Case 1:} $2<q<\overline{p}$. Since $u_{\mu}$ is a  positive ground state solution of $E_{\mu}(u)$ on $\{u\in S_{a}: ||u_{\mu}||^{2}_{D_{s}(\mathbb{R}^{N})}< R_{0}\}$, where $R_{0}(a,\mu)$ is defined by  Lemma \ref{Lem2},  since $R_{0}$ is defined by $h(R_{0})=0$, see $h$ in \eqref{equ5},  we can check that $R_{0}=R_{0}(a,\mu)\rightarrow 0$  as $\mu\rightarrow 0^{+}$, thus $||u_{\mu}||^{2}_{D_{s}(\mathbb{R}^{N})}< R_{0}\rightarrow 0$ as $\mu\rightarrow 0^{+}$.
Since for every $u\in S_{a}$, by fractional Gagliardo-Nirenberg-Sobolev inequality \eqref{U3} and Sobolev inequality \eqref{int2}
\begin{align*}
0>m(a,\mu)=E_{\mu}(u_{\mu})\geq\frac{1}{2}||u_{\mu}||^{2}_{D_{s}(\mathbb{R}^{N})}
-\frac{\mu}{q}C^{q}_{N,q,s}||u_{\mu}||^{q\gamma_{q,s}}_{D_{s}(\mathbb{R}^{N})}a^{q(1-\gamma_{p,s})}-\frac{1}{2_{s}^{\ast}}S^{-\frac{2_{s}^{\ast}}{2}}_{s}||u_{\mu}||^{2_{s}^{\ast}}_{D_{s}(\mathbb{R}^{N})}\rightarrow 0
\end{align*}
as $\mu\rightarrow 0^{+}$.

\smallskip

{\bf Case 2:} $\overline{p}\leq q<2_{s}^{\ast}$. Let $\widetilde{\mu}\geq 0$ and \eqref{equ4} holds. Firstly, we show that the family of positive radial ground states $\{u_{\mu}: 0<\mu<\widetilde{\mu}\}$ is a bounded in $H^{s}(\mathbb{R}^{N})$. If $q=\overline{p}=2+4s/N$, then by Lemma \ref{Lem19} and $P_{\mu}(u_{\mu})=0$, we have
\begin{align*}
m(a,0)\geq m(a,\mu)=E_{\mu}(u_{\mu})&=\frac{s}{N}\left(||u_{\mu}||^{2}_{D_{s}(\mathbb{R}^{N})}-\frac{2\mu}{\overline{p}}\int_{\mathbb{R}^{N}} |u|^{\overline{p}}dx\right)\\
&\geq \frac{s}{N}\left(1-\frac{2\mu}{\overline{p}}C^{\overline{p}}_{N,\overline{p},s}a^{\frac{4s}{N}}\right)||u||^{2}_{D_{s}(\mathbb{R}^{N})}.
\end{align*}
If $\overline{p}<q<2_{s}^{\ast}$, by the similar arguments as above, we have
\begin{align*}
m(a,0)\geq m(a,\mu)=E_{\mu}(u_{\mu})=\frac{s}{N}\int_{\mathbb{R}^{N}} |u_{\mu}|^{2_{s}^{\ast}}dx+\frac{\mu}{q}\left(\frac{q\gamma_{q,s}}{2}-1\right)\int_{\mathbb{R}^{N}} |u_{\mu}|^{q}dx.
\end{align*}
Thus, $\{u_{\mu}\}$ is bounded in $L^{q}(\mathbb{R}^{N})\cap L^{2_{s}^{\ast}}(\mathbb{R}^{N}) $. From $P_{\mu}(u_{\mu})=0,$ we also have $\{u_{\mu}\}$ is bounded in $H^{s}(\mathbb{R}^{N}).$ Since $$\widetilde{\lambda}_{\mu}a^{2}=||u_{\mu}||^{2}_{D_{s}(\mathbb{R}^{N})}-\mu\int_{\mathbb{R}^{N}} |u_{\mu}|^{q}dx-\int_{\mathbb{R}^{N}} |u_{\mu}|^{2_{s}^{\ast}}dx=\mu(\gamma_{q,s}-1)\int_{\mathbb{R}^{N}} |u_{\mu}|^{q}dx\rightarrow 0$$ as $\mu\rightarrow 0^{+}$. Therefore $u_{\mu}\rightharpoonup u$ weakly in $H^{s}(\mathbb{R}^{N}),\ D_{s}(\mathbb{R}^{N}),\ L^{2_{s}^{\ast}}(\mathbb{R}^{N})$ and $u_{\mu}\rightharpoonup u$ strongly in $L^{q}(\mathbb{R}^{N})$, $\widetilde{\lambda}_{\mu}\rightarrow 0$. Let $||u_{\mu}||^{2}_{D_{s}(\mathbb{R}^{N})}\rightarrow \ell\geq 0$, if $\ell=0$, then $u_{\mu}\rightarrow 0$ strongly in $D_{s}(\mathbb{R}^{N})$, so $E_{\mu}(u_{\mu})\rightarrow 0.$ However, by Lemma \ref{Lem19}, we know that $E_{\mu}(u_{\mu})\geq m(a,\widetilde{\mu})>0$ for every $0<\mu<\widetilde{\mu}$, a contradiction. Thus $\ell>0$. Since $P_{\mu}(u_{\mu})=0$, we have $$\int_{\mathbb{R}^{N}} |u_{\mu}|^{2_{s}^{\ast}}dx=||u_{\mu}||^{2}_{D_{s}(\mathbb{R}^{N})}-\mu\gamma_{q,s}\int_{\mathbb{R}^{N}} |u_{\mu}|^{q}dx\rightarrow \ell,\ \ \text{as}\ \mu\rightarrow 0^{+}.$$ Therefore, by the Sobolev embedding $\ell\geq S_{s}\ell^{\frac{2}{2_{s}^{\ast}}}$, which implies that $\ell\geq S^{\frac{N}{2s}}_{s}$. On the other hand, we have
$$\frac{\ell}{N}=\lim_{\mu\rightarrow 0^{+}}\left[\frac{s}{N}||u_{\mu}||^{2}_{D_{s}(\mathbb{R}^{N})}-\frac{\mu}{q}\left(1-\frac{q\gamma_{q,s}}{2_{s}^{\ast}}\right)\int_{\mathbb{R}^{N}} |u_{\mu}|^{q}dx\right]=\lim_{\mu\rightarrow 0^{+}}E_{\mu}(u)\leq m(a,0)=\frac{s}{N}S^{\frac{N}{2s}}_{s}.$$
Thus, $\ell=S^{\frac{N}{2s}}_{s}$ and the desired conclusion follows.
\end{proof}
\section*{Acknowledgements}
B. Zhang was supported by the
National Natural Science Foundation of China (No. 11871199), the Heilongjiang Province Postdoctoral
Startup Foundation, PR China (LBH-Q18109), and the Cultivation Project of Young and Innovative Talents
in Universities of Shandong Province.

\end{document}